\pdfoutput=1

\documentclass[12pt,a4paper]{article}

\usepackage{amsmath,amssymb}
\usepackage{amsthm}
\usepackage{enumerate}
\usepackage{ifxetex,ifluatex}
\usepackage{calc}
\makeatletter
\def\@hspace#1{\begingroup\setlength\dimen@{#1}\hskip\dimen@\endgroup}
\makeatother

\newif\ifxetexorluatex
\ifxetex
  \xetexorluatextrue
\else
  \ifluatex
    \xetexorluatextrue
  \else
    \xetexorluatexfalse
  \fi
\fi

\usepackage[pdftex]{hyperref}

\usepackage[square,compress,comma, numbers,sort]{natbib}

\usepackage{tikz}
\usetikzlibrary{calc,intersections}

\usepackage{xcolor}
\definecolor{lw}{RGB}{0,0,255}
\definecolor{todo}{RGB}{0,0,255}

\newtheorem{thm}{Theorem}[section]
\newtheorem{prop}[thm]{Proposition}

\newtheorem{lemm}[thm]{Lemma}
\newtheorem{cor}[thm]{Corollary}
\newtheorem{defn}[thm]{Definition}

\newtheorem{remark}[thm]{Remark}

\newtheorem{conj}[thm]{Conjecture}
\numberwithin{equation}{section}

\newtheorem{taggedtheoremx}{Theorem}
\newenvironment{taggedtheorem}[1]
 {\renewcommand\thetaggedtheoremx{#1}\taggedtheoremx}
 {\endtaggedtheoremx}

\allowdisplaybreaks[4]

\def\R{{\mathbb{R}}}
\def\E{{\mathbb{E}}}
\def\P{{\mathbb{P}}}
\def\H{{\mathbb{H}}}

\def\F{{\mathbb{F}}}

\def\N{{\mathbb{N}}}
\def\bS{{\mathbb{S}}}
\def\T{{\mathbb{T}}}
\def\Z{{\mathbb{Z}}}

\def\sA{{\mathcal{A}}}

\def\sC{{\mathcal{C}}}

\def\sG{{\mathcal{G}}}
\def\sH{{\mathcal{H}}}
\def\sL{{\mathcal{L}}}

\def\sP{{\mathcal{P}}}

\def\sT{{\mathcal{T}}}

\def\e{{\mathrm{e}}}
\def\rd{{\mathrm{d}}}

\def\id{{\mathbf{1}}}
\def\Pr{{\mathrm{Pr}}}
\def\BRW{{\text{BRW}}}
\def\brw{{\mathrm{BRW}}}

\DeclareMathOperator*\diam{diam}

\usepackage{xifthen}
\newcommand{\gp}[3][]{%
  \ifthenelse{\equal{#1}{}}{\left( {#2\,|\, #3} \right)}{\left( {#2\,|\,#3} \right)_{#1}}%
}

\makeatletter
\newcommand*{\rmnum}[1]{\expandafter\@slowromancap\romannumeral #1@}
\makeatother

\usepackage{filecontents}

\title{Limit set of branching random walks on hyperbolic groups}
\author{}

\begin{document}

\begin{center}
  {\large\bf Limit set of branching random walks on hyperbolic groups\footnote{The project is supported partially by CNNSF (No. 11671216) and by Hu Xiang Gao Ceng Ci Ren Cai Ju Jiao Gong Cheng-Chuang Xin Ren Cai (No. 2019RS1057).}}

  \vskip 2mm

  Vladas Sidoravicius, Longmin Wang \& Kainan Xiang

\end{center}

\bigskip

\noindent\textbf{Abstract}  Let $\Gamma$ be a nonelementary hyperbolic group with a word metric $d$ and $\partial\Gamma$ its hyperbolic boundary equipped with a visual metric $d_a$ for some parameter $a>1$. Fix a superexponential symmetric probability $\mu$ on $\Gamma$ whose support generates $\Gamma$ as a semigroup, and denote by $\rho$ the spectral radius of the random walk $Y$ on $\Gamma$ with step distribution $\mu$. Let $\nu$ be a probability on $\{1,\, 2, \, 3, \, \ldots\}$ with mean $\lambda=\sum\limits_{k=1}^\infty k\nu(k)<\infty$.
  Let $\BRW(\Gamma, \, \nu, \, \mu)$ be the branching random walk on $\Gamma$ with offspring distribution $\nu$ and base motion $Y$ and $H(\lambda)$ the volume growth rate for the trace of $\BRW(\Gamma, \, \nu, \, \mu)$. We prove for $\lambda \in [1, \, \rho^{-1})$ that the Hausdorff dimension of the limit set $\Lambda$, which is the random subset of $(\partial \Gamma, \, d_a)$ consisting of all accumulation points of the trace of $\BRW(\Gamma, \, \nu, \, \mu)$, is given by $\log_a H(\lambda)$. Furthermore, we prove that $H(\lambda)$ is almost surely a deterministic, strictly increasing and continuous function of $\lambda \in [1, \, \rho^{-1}]$, is bounded by the square root of the volume growth rate of $\Gamma$, and has critical exponent $1/2$ at $\rho^{-1}$ in the sense that
  \[
    H(\rho^{-1}) - H(\lambda) \sim C \sqrt{\rho^{-1} - \lambda} \quad \text{as } \lambda \uparrow \rho^{-1}
  \]
  for some positive constant $C$. We conjecture that the Hausdorff dimension of $\Lambda$ in the critical case $\lambda=\rho^{-1}$ is $\log_aH(\rho^{-1})$ almost surely. This has been confirmed on free groups or the free product (by amalgamation) of finitely many finite groups equipped with the word metric $d$ defined by the standard generating set.



\tableofcontents

\section{Introduction and main results}
\label{s:intro}

Let $G = (V, \, E)$ be a locally finite, connected infinite graph and consider a branching random walk (BRW) on $G$ described as follows. We begin with one particle at $o \in V$ at time $0$. For $n \geq 1$, each particle alive at time $n$ dies and gives birth to an independent random number of offspring particles according to some probability measure $\nu$, each of which independently takes a step according to a random walk on $G$ with transition probabilities $P$. When $G$ is a Cayley graph of a finitely generated group $\Gamma$ with respect to a symmetric generating set $S$, the transition probabilities $P$ may be defined by a symmetric probability measure $\mu$ on $\Gamma$, that is, $P(x,\, y) = \mu(x^{-1} y)$ for $x$, $y \in \Gamma$. In this case, the base random walk $Y$ can be obtained by multiplying random elements of $\Gamma$ distributed independently according to $\mu$. We write  ${\rm BRW}(\Gamma, \, \nu, \, \mu)$ for the corresponding BRW. When the underlying motion is the simple random walk, we call it the simple BRW on $G$ or $\Gamma$. In this paper, we always assume that $\nu(0)=0$ to guarantee the BRW survives almost surely, and
\[
  \lambda = \sum\limits_{k=1}^{\infty} k \nu(k) \in [1, \, \infty).
\]

A natural question is to ask whether the BRW eventually fills up the whole graph. This is equivalent to the question of the recurrence of the process. A BRW is said to be \emph{recurrent} if, with positive probability (in fact, with probability $1$), some (and hence all) vertex of $G$ is visited by infinitely many particles of the BRW, and \emph{transient} otherwise. There is a general criteria for transience and recurrence in terms of the mean offspring $\lambda$ and the spectral radius $\rho$ of the underlying random walk: the BRW is transient if and only if $\lambda \leq \rho^{-1}$; see Benjamini and Peres~\cite{BP94} for the case $\lambda \neq \rho^{-1}$, and Gantert and M{\"u}ller~\cite{GM06} for the critical case $\lambda = \rho^{-1}$. In the transient case the \emph{trace} of the BRW, i.e., the subgraph consisting of visited vertices and traversed edges, is a proper random subgraph of the original graph. Benjamini and M\"uller~\cite{BM12} studied first general qualitative properties of the trace of BRW on groups; in particular, they proved that it has an exponential volume growth in general. However, their approach gives no quantitative results on the growth rate. Very recently, Hutchcroft proved in~\cite{H19} that any two independent transient BRWs on a nonamenable group intersect at most finitely often almost surely; see also~\cite{CR15} for the case of the subcritical BRWs and~\cite{GM17} for the critical BRWs on planar hyperbolic groups. As a consequence, the trace of such a BRW is tree-like in the sense that it has infinitely many ends almost surely. Thus it is quite interesting to study the geometric property and provide quantitative estimates for the trace of the BRW in the transient regime.

Let $\F_q$ be the free group over $q \geq 2$ symbols $s_1$, $\ldots$, $s_q$. The Cayley graph of $F_q$ with respect to generating set $S = \{s_1,\, s_1^{-1}, \, \ldots,\, s_q,\, s_q^{-1}\}$ is the $(2q)$-regular tree $\T_{2q}$. 
On $\F_q$ the geometry of the trace of the BRW has been well understood in the literature when
the underlying motion is a (lazy) nearest-neighbor random walk. The boundary $\partial \F_q$ 
is defined to be the set of semi-infinite reduced words from the generating set $S$. For each real number $a > 1$, there is a natural metric $d_a$ on $\partial \F_q$ 
defined by
\begin{eqnarray}\label{e:metric-free}
d_a(\omega,\, \omega') = a^{- N(\omega,\, \omega')},
\end{eqnarray}
where $N(\omega,\, \omega')$ is the maximum integer $n$ such that the sequences $\omega$ and $\omega'$ agree in entries $1$, $2$, $\ldots$, $n$. Under this metric, $\partial \F_q$ 
is a compact metric space and its Hausdorff dimension $\dim_H(\partial \F_q) = \log_a (2q-1)$. 
Let $\mu$ be a symmetric probability measure $S$ or $S\cup\{e\}$ and $\nu$ a probability on $\N$ with mean $\lambda\in [1, \, \rho^{-1}]$. 
Here $e$ is the identity element of $\F_q$. 
Note that under this setting the base motion $Y$ is a (lazy) nearest-neighbor random walk. As stated before, if $\lambda \in [1,\, \rho^{-1}]$, then the particles in the BRW almost surely eventually vacates every finite subset of $\F_q$, 
and thus the accumulation points at infinity form a random closed subset $\Lambda$ of $\partial \F_q$.  
We call $\Lambda$ the limit set of the BRW and denote by $\dim_H(\Lambda)$ its Hausdorff dimension under the metric $d_a$. 
Hueter and Lalley~\cite{HL00} provided a quantitative description of the Hausdorff dimension of $\Lambda$ which we restate in Theorem~\ref{T:HL} as follows. In fact, their results holds for all regular tree $\T_d$ of degree $d \geq 3$ (i.e., the Cayley graph of the free product $(\Z_2)^{*d}$). The special case when $Y$ is a simple random walk was already proved in Liggett~\cite{L96b}. We remark that in their setting and notion, weak survival is equivalent to transience in our language.

\begin{taggedtheorem}{A}[Hueter and Lalley {\cite{HL00}}, Lalley {\cite{L06}}]
  \label{T:HL}
  Consider a BRW on $\F_q$ with mean offspring $\lambda \in [1, \, \rho^{-1}]$ and let $M_n$ be the number of points at distance $n$ from the root that are ever visited by particles of the BRW. Then almost surely $H(\lambda) = \limsup\limits_{n \to \infty} M_n^{1/n}$ exists and is a constant, and the limit set $(\Lambda, \, d_a)$ has Hausdorff dimension $h(\lambda) = \log_a \theta(\lambda)$. The functions $\theta(\lambda)$ and $h(\lambda)$ are continuous and strictly increasing in $\lambda \in [1,\, \rho^{-1}]$, and have critical exponent $1/2$ at $\rho^{-1}$ in the sense that for some constant $C > 0$, 
  \begin{equation}
    \label{e:thetalambda}
    \theta\left(\rho^{-1}\right) - \theta(\lambda) \sim C \, \sqrt{\rho^{-1} - \lambda},\quad h\left(\rho^{-1}\right) - h(\lambda) \sim \frac{C}{\theta(\rho^{-1})\log a} \, \sqrt{\rho^{-1} - \lambda}
  \end{equation}
  as  $\lambda \uparrow \rho^{-1}$. 
  Furthermore, we have that 
  \begin{equation}
    \label{e:thetaR}
    \theta(\lambda) \leq \sqrt{2 q - 1}, \quad
    \text{and }
    h(\lambda) \leq \frac{1}{2} \dim_H \left( \partial \F_q \right), 
  \end{equation}
  with equality holding if and only if the underlying random walk is a lazy simple random walk and $\lambda = \rho^{-1}$. 
\end{taggedtheorem}


Recently, Candellero, Gilch and M\"{u}ller~\cite{CGM12} extended Theorem~\ref{T:HL} to the BRW on free product (by amalgamation) of finite groups when $\nu$ has a finite second moment and base motion $Y$ is a lazy nearest-neighbor random walk, 
with the upper bound in~\eqref{e:thetaR} replaced by the square root of volume growth rate of the underlying group. In fact, they obtained more general results on free product of graphs such as the critical exponent of Hausdorff dimension of $\Lambda$ being $1$ or $1/2$ depends on whether the weighted Green function is left differentiable at $\rho^{-1}$ or not.

The same type of phase transition on the Hausdorff dimension of the geometric boundaries has also been observed for other growth models including the branching Brownian motion on real hyperbolic spaces $\H^m$ (see Lalley and Sellke~\cite{LS97} for $m=2$ and Karpelevich, Pechersky and Suhov~\cite{KPS98} for $m \geq 3$), branching diffusion on hyperbolic spaces (Kelbert and Suhov~\cite{KS07}) and the isotropic contact processes on regular trees (Liggett~\cite{L96b}, Lalley and Sellke~\cite{LS98}). Other instances of processes for which the weak survival phase (or the coexistence phase, i.e., the phase that infinitely many connected clusters exist almost surely) is known to exist include percolation on nonamenable hyperbolic quasi-transitive graphs~\cite{H19b} (see also~\cite{L96,BS01,L01} for the case of transitive nonamenable planar hyperbolic graphs).

In~\cite{HL00}, the function $\theta(\lambda)$ was expressed as the unique solution to an algebraic equation in terms of the weighted Green functions of the base random walk. Based on this formula, Lalley~\cite{L06} pointed out that the critical exponent $1/2$ in~\eqref{e:thetalambda} is related to the exponent $3/2$ occurring in the following asymptotics obtained in~\cite{GW86}: for any $x$ and $y$, there is $C_{x, \, y} > 0$ such that  
\begin{equation}
  \label{e:hkdecay}
  p_n(x,\, y)\sim C_{x,\, y} \, \rho^n \, n^{-3/2}
\end{equation}
as $n$ tends to infinity along the set $\{n: \ p_n(x, \, y) > 0\}$, where $p_n(x,\, y):=\mathbb{P}(Y_n=y\, \vert\, Y_0=x)$. He conjectured that the BRWs on the groups with the property~\eqref{e:hkdecay} should have the same type of phase transition as that in Theorem~\ref{T:HL}. Recently, Gou\"ezel~\cite{G14, G15} proved such a heat kernel decay~\eqref{e:hkdecay} holds for any nonelementary hyperbolic groups and random walks with superexponential step distribution. So it is natural to expect that the BRW on nonelementary hyperbolic groups has critical exponent $1/2$ for both the volume growth rate of its trace and the Hausdroff dimension of its limit set. It is well-known that the importance of critical exponents is based largely on the universality principle, which plays an important role in mathematical physics. 

To state our main results, we first review some standard facts on (Gromov) hyperbolic groups; see~\S\ref{s:hg} or~\cite{G87,GH90} for more details. Let $\Gamma$ be a finitely generated group equipped with a word metric $d$. Let $\bS_n := \{x \in \Gamma:\ d(e,\, x) = n \}$ the sphere of radius $n$ centered at the identity element $e$. 
The \emph{exponential growth rate} (or \emph{volume growth rate}) of $\Gamma$ is defined by $\limsup_{n\rightarrow\infty} \left| \bS_n \right|^{1/n}$, where $|A|$ is the cardinality of a set $A$. The logarithm of the volume growth rate is called the \emph{volume entropy}, i.e.,
\begin{equation}
  \label{e:vrate1}
  v = \limsup\limits_{n \to \infty} \frac{1}{n} \log |\bS_n|.
\end{equation}

For $x$, $y$, $z \in \Gamma$, define the \emph{Gromov product} $\gp[x]{y}{z}$ of $y$ and $z$ with respect to $x$ by
\[
  \gp[x]{y}{z} := \frac{1}{2} \left\{ d(x,\, y) + d(x,\, z) - d(y,\, z) \right\}.
\]
The group $\Gamma$ is called \emph{(Gromov) hyperbolic} if there is a constant $\delta \geq 0$ such that
\[
  \gp[w]{x}{y} \geq \min \left\{ \gp[w]{x}{z},\, \gp[w]{y}{z} \right\} - \delta
\]
for all $x$, $y$, $z$ and $w \in \Gamma$.

Now assume that the hyperbolic group $\Gamma$ is nonelementary, that is, $\Gamma$ is not finite nor virtually $\Z$. Let $\mu$ be an admissible probability measure on $\Gamma$. Here a measure is called \emph{admissible} if its support generates $\Gamma$ as a semigroup. Throughout this paper, we assume that $\mu$ is \emph{superexponential} in the sense that $\sum_{x \in \Gamma} \e^{r d(e,\, x)} \mu (x) < \infty$ for all $r \geq 0$, and
is \emph{symmetric} in the sense that $\mu(x^{-1})=\mu(x)$ for all $x\in\Gamma$. Let $Y=(Y_n)_{n=0}^{\infty}$ be the random walk with step distribution $\mu$ and $p_n(x,\, y) = \P \left( Y_n = y \, |\, Y_0 = x \right)$ the associated heat kernel. For $r \geq 0$ and $x$, $y \in \Gamma$, define the \emph{(weighted) Green function}
$$G_r(x,\, y) = \sum_{n=0}^{\infty} r^n p_n(x,\, y).$$
The \emph{spectral radius} $\rho$ is the reciprocal of the convergence radius of the series $G_r(x,\, y)$ and is independent of the choices for $x$ and $y$. Since $\Gamma$ is nonelementary, we have that $\rho \in (0,\, 1)$. Furthermore, $G_{\rho^{-1}}(x,\, y)$ is finite for any $x$, $y \in \Gamma$; see~\cite[Theorem II.7.8]{W00}. Set
\[
  H_n(r) := \sum_{x \in \bS_n} G_r(e,\, x) \quad \text{and}\  H(r) := \limsup_{n \to \infty} \left\{ H_n(r) \right\}^{1/n}\quad \text{for}\ r \leq \rho^{-1}.
\]

Recall that $\nu$ is a probability measure on $\N=\{1,\, 2,\, \ldots\}$ with mean $\lambda \in [1,\, \infty)$. Consider the ${\rm BRW}(\Gamma,\, \nu, \, \mu)$. Let $\sP$ be the set of points in $\Gamma$ that are ever visited by the
${\rm BRW}(\Gamma, \, \nu, \, \mu)$. Denote by $M_n$ be the number of points $x \in \sP$ with $|x| = n$. Our first main result is on the growth rate for $M_n$ in the transient regime $\lambda \in [1,\, \rho^{-1}]$. As stated before, if $\lambda > \rho^{-1}$, then the ${\rm BRW}(\Gamma, \, \nu, \, \mu)$ is recurrent and hence $M_n = |\bS_n|$.

\begin{thm}
  \label{T:main1}
Assume $\mu$ is an admissible, superexponential, symmetric probability on a nonelementary hyperbolic group $\Gamma$, and $\nu$ is a probability on $\N$ with mean $\lambda \in [1,\, \rho^{-1}]$. Consider the ${\rm BRW}(\Gamma, \, \nu, \, \mu)$ starting at $e$. Then
  \[
    H(\lambda) = \limsup_{n \to \infty} M_n^{1/n} \quad \text{a.s.}
  \]
  Furthermore, $H(\lambda)\in\left[1,\, \e^{v/2}\right]$ is continuous and strictly increasing in $\lambda \in [1,\, \rho^{-1}]$, and has critical exponent $1/2$ at $\rho^{-1}$ in the sense that for some constant $C > 0$,
  \begin{equation}
    \label{e:Phicriticle}
    H(\rho^{-1}) - H(\lambda) \sim C \sqrt{\rho^{-1} - \lambda} \quad \text{as\ } \lambda \uparrow \rho^{-1}.
  \end{equation}
\end{thm}

Let $\partial \Gamma$ be the hyperbolic boundary of the metric space $(\Gamma, \, d)$. The Gromov product can be naturally extended to $\partial \Gamma$. A metric $d_a$ on $\partial \Gamma$ is said to be \emph{visual} with parameter $a > 1$ if there is a positive constant $c$ such that
\[
  c^{-1} \, a^{- \gp[e]{\xi}{\xi'}} \leq d_a(\xi,\, \xi') \leq c \, a^{- \gp[e]{\xi}{\xi'}}
\]
for all $\xi$, $\xi' \in \partial \Gamma$. Note that the visual metric $d_a$ exists for $a-1>0$ sufficiently small, and $(\partial \Gamma,\, d_a)$ is a compact metric space. Furthermore, any two visual metrics (not necessarily with the same parameters) define the same topology. By~\cite{C93}, the Hausdorff dimension $\dim_H(\partial\Gamma)$ of $(\partial\Gamma, \, d_a)$ is given by $v / \log a$. 

The accumulation points of the set $\sP$ in $\Gamma \cup \partial \Gamma$ form a random closed subset $\Lambda \subset \partial \Gamma$, which is called the \emph{limit set} of ${\rm BRW}(\Gamma,\, \, \nu,\, \, \mu)$.
As stated before, when $\lambda > \rho^{-1}$, ${\rm BRW}(\Gamma,\, \nu,\, \mu)$ is recurrent and hence $\Lambda = \partial \Gamma$ a.s. In this paper, we prove in the regime $\lambda \in [1,\, \rho^{-1})$ that the Hausdorff dimension of $\Lambda$ is the logarithm of the volume growth rate for the BRW. The critical case $\lambda = \rho^{-1}$ remains open.

\begin{thm}
  \label{thm1.1}
  Assume $\lambda \in [1, \, \rho^{-1})$. Let $\mu$ be an admissible, superexponential, symmetric probability on a nonelementary hyperbolic group $\Gamma$, and $\nu$ a probability on $\N$ with mean $\lambda$.
  Then the Hausdorff dimension $\dim_H(\Lambda)$ of the limit set $(\Lambda, \, d_a)$ of $\BRW(\Gamma, \, \nu, \, \mu)$ is $h(\lambda) = \log_a H(\lambda)$ almost surely. In particular, $h(\lambda)$ is continuous and strictly increasing in $\lambda \in [1, \, \rho^{-1})$, $h(\rho^{-1}-) \leq \frac{1}{2} \dim_H (\partial \Gamma)$, and has critical exponent $1/2$ at $\rho^{-1}$ in the sense that
  \[
    h(\rho^{-1} - ) - h(\lambda) \sim \frac{C}{H(\rho^{-1}) \log a} \, \sqrt{\rho^{-1} - \lambda} \quad \text{as } \rho \uparrow \rho^{-1}, 
  \]
  where $C$ is the constant in Theorem~\ref{T:main1}. 
\end{thm}

\begin{remark}
  \begin{enumerate}[(i)]

%
  \item As we will see in the proof of Theorem~\ref{T:Hr} that, for $\lambda < \rho^{-1}$, the inequality $\log H(\lambda) \leq v / 2$ holds on any nonamenable groups. See~\cite[\S 8]{LS97} for an explanation why $1/2$ appears here, using the ``backscattering principle''.

  \item The critical exponent $1/2$ is universal among nonelementary hyperbolic groups in the sense that it does not depend on the particular hyperbolic groups, the offspring distributions and the base motion of the BRWs. On general nonamenable groups, the exponent is not necessarily $1/2$. For example, it is proved in~\cite[Theorem 3.10]{CGM12} that, on free product of (not necessarily finite) groups the critical exponent is $1$ or $1/2$ according to whether the Green function $G_\lambda(e,\, e)$ is differentiable (from left) at the critical point $\lambda = \rho^{-1}$ or not.

  \item 
  When $\nu(0) > 0$, i.e., the particles have zero offspring with positive probability, by slightly modifications, one can show that Theorems~\ref{T:main1} and~\ref{thm1.1} hold on the event that $\brw(\Gamma, \, \nu, \, \mu)$ survives.
  \end{enumerate}
\end{remark}

\vskip1em


In the critical case $\lambda = \rho^{-1}$, it is clear that $\dim_H (\Lambda) \geq \log_a H(\rho^{-1})$. We conjecture it is indeed an equality. 
\begin{conj}
\label{Conjecture}
Let $\mu$ be an admissible, superexponential, symmetric probability on a nonelementary hyperbolic group $\Gamma$, and $\nu$ a probability on $\N$ with mean $\lambda=\rho^{-1}$. Then
the limit set $(\Lambda, \, d_a)$ of ${\rm BRW}(\Gamma,\,\nu,\,\mu)$ has Hausdorff dimension $\log_a H(\rho^{-1})$ almost surely.
\end{conj}

When $\Gamma$ is a free group with at least 2 generators or a free product (by amalgamation) of finitely many finite groups, equipped with the natural canonical word metric,
Conjecture~\ref{Conjecture} holds true. Its proof depends heavily on the tree structure or the block tree structure of the corresponding Cayley graphs; see Remark~\ref{R:tree-conj} for more details.

\bigskip

We describe briefly the proofs of Theorems~\ref{T:main1} and~\ref{thm1.1} as follows.
\begin{enumerate}[(1)]
\item Our proofs rely essentially on the so-called Ancona's inequalities~\cite{A88,G14,G15}, which state that the Green function $G_\lambda(x, \, y)$ is roughly multiplicative along geodesics. In particular, using Ancona's inequalities, we can prove that $H_n(\lambda) = \sum_{x:\ d(e,\, x) = n} G_\lambda(e,\, x)$ is roughly sub- and super-multiplicative and has purely exponential growth. This kind of behavior is typical on nonelementary hyperbolic groups and Riemannian manifolds with negative curvature, and is closely related to the critical exponent $1/2$.
\item We use the first and second moment methods to prove that the growth rate for $M_n$ coincides with that of the sum $H_n(\lambda)$ of Green functions over the sphere $\bS_n$. The estimate for the two-point correlation function of the BRW is required to get the upper bound for the second moment of $M_n$ and has its own interests. Our estimates on the correlation is not optimal and may be significantly improved.
\item We prove that in the subcritical case $\lambda < \rho^{-1}$, along each geological path $\gamma$ the BRW has a positive speed and thus converges to a unique limit point $X_\gamma$ in the hyperbolic boundary $\partial \Gamma$, and further the discrepancy of a particle from the geodesic from $e$ to $X_\gamma$ is sufficiently small relative to its generation. Moreover, the map $\gamma \mapsto X_\gamma$ is continuous and the collection of all such points $X_\gamma$ is exactly the limit set $\Lambda$ of the BRW. These facts lead to a sequence of nice coverings of $\Lambda$ using the concept of shadows introduced by Sullivan~\cite{S79} and the upper bound $\dim_H(\Lambda) \leq \log_a H(\lambda)$ for the Hausdorff dimension follows. Using these ingredients, we may get further geometric properties of the limit set $\Lambda$. For examples, one can prove that $\Lambda$ is totally disconnected and that two independent subcritical BRWs  have disjoint limit sets. The latter statement is much stronger than that their traces have finitely many intersections.

  Let $\chi_n$ be the uniform distribution on $\sP_n$ and $\chi$ a weak limit point conditioned on the configuration of the BRW. Using the estimate for the second moment of $M_n = |\sP_n|$, we can prove that $\chi$ is supported on $\Lambda$ and satisfies
  \[
    \int_\Lambda \int_\Lambda \left[ d_a(x,\, y) \right]^{-h} \rd \chi(x) \rd \chi(y) < \infty
  \]
  for any $h < \log_a H(\lambda)$. The lower bound $\dim_H(\Lambda) \geq \log_a H(\lambda)$ follows from the Frostman's Lemma.
\item Using the Cannon automaton coding geodesics on a nonelementary hyperbolic group and thermodynamic formalism of the resulting symbolic dynamics, we may express $\log H(\lambda)$ as the pressure of certain transfer operators of the dynamical system. Then we show that the critical exponent for $H(\lambda)$ is $1/2$ by applying the perturbation method developed in Gou\"ezel~\cite{G14}.
\end{enumerate}

Throughout this paper, we will use $C$, $C_1$, $C_2$, $\ldots$, and $c$, $c_1$, $c_2$, $\ldots$ to denote positive constants, whose precise values are not important and may be changed from line to line. The remainder of the paper is organized as follows. In~\S\ref{s:hg}, we review some preliminaries on hyperbolic groups and (branching) random walks thereon. The proofs of Theorems~\ref{T:main1} and~\ref{thm1.1} will be divided into four steps as described above, each of them is presented in one of the following four sections~\S\ref{s:rw}-\S\ref{s:exponent}. 

\section{Preliminaries}
\label{s:hg}
In this section we review some necessary preliminaries on random walks and branching random walks on hyperbolic groups. The readers are referred to~\cite{G87,GH90,CP93} for a general introduction to hyperbolic groups, and to~\cite{W00,K00,BHM11,G14,G15,MS20} for an introduction and recent progress of random walks on hyperbolic groups/spaces.

\subsection{Hyperbolic groups}

Let $\Gamma$ be a finitely generated group with the identity element $e$ and a finite symmetric generating set $S$. The word distance $d = d_S$ is defined by
\[
  d(x, \, y) =
  \begin{cases}
    \inf \left\{n:\ \exists s_1, \, \ldots, \, s_n \in S \text{ with\ } x^{-1}y = s_1 \cdots s_n\right\}, & x\not=y, \\
     0,              &x=y.
  \end{cases}
\]
This is the graph distance on the Cayley graph of $(\Gamma,\, S)$. Write $|x|:=d(e,\, x)$ for simplicity. For $x$, $y \in \Gamma$, let $[x, \, y]$ be an arbitrarily chosen geodesic segment in $\Gamma$ connecting $x$ and $y$.

The \emph{Gromov (inner) product} of $x$ and $y$ with respect to $z$ is defined as
  \[
    \gp[z]{x}{y} := \frac{1}{2} \left[ d(x,\, z) + d(y,\, z) - d(x,\, y) \right].
  \]
When $z = e$, write $\gp{x}{y} := \gp[e]{x}{y}$.

\begin{defn}
 Let $\delta\geq 0$. The group $\Gamma$ is $\delta$-\emph{hyperbolic (in the sense of Gromov)} if, for any $w$, $x$, $y$, $z \in \Gamma$, the following ultrametric type inequality holds
  \begin{equation}
    \label{e:gromov}
    \gp[w]{x}{y} \geq \min \left\{ \gp[w]{x}{z},\, \gp[w]{y}{z} \right\} - \delta.
  \end{equation}
Say $\Gamma$ is (Gromov) hyperbolic if it is $\delta$-hyperbolic for some $\delta \geq 0$.
\end{defn}

Note that~\eqref{e:gromov} can be rewritten as
\begin{equation}
  \label{e:4points}
  d(x,\, y) + d(z,\, w) \leq \max \left\{ d(x,\, z) + d(y,\, w),\ d(x, \, w) + d(y,\, z) \right\} + 2 \delta
\end{equation}
and interpreted as follows. There are three possible ways to divided the four points into pairs and the corresponding sums of distances are $g=d(x,\, y) + d(z,\, w),\ m=d(x,\, z) + d(y,\, w),\ s=d(x, \, w) + d(y,\, z)$. Assume that $s\leq m\leq g$ (rename the points if necessary). Then~\eqref{e:4points} is written as $g\leq m+2\delta$, that is, the greatest sum cannot exceed the mean sum by more than $2\delta$. 

There are other equivalent definitions of hyperbolicity. For example, a group $\Gamma$ is called \emph{$\delta$-hyperbolic in the sense of Rips} if all geodesic triangles in $(\Gamma,\, d)$ are $\delta$-thin, that is, each side of such a triangle is contained in the $\delta$-neighborhood of the union of the other two sides. Intuitively, any geodesic triangle is roughly isometric to a ``tripode''. In general, from~\cite[Theorem 2.12]{GH90}, there is a constant $C$ depending only on $n$ and $\delta$ such that for any subset $A$ consisting of at most $n$ points, there exists a map $\Psi$ from $A$ to a metric tree such that
\[
  d(x, \, y)- C \leq d(\Psi(x), \, \Psi(y)) \leq d(x, \, y) \text{ for any } x, \, y\in A.
\]
The Rips hyperbolicity and Gromov hyperbolicity are equivalent. More precisely,
\begin{enumerate}[(i)]
\item if a group is $\delta$-hyperbolic, then it is ($4\delta$)-hyperbolic in the sense Rips; 
\item if a group is $\delta$-hyperbolic in the sense of Rips, then it is $\delta$-hyperbolic. 
\end{enumerate}
See for example~\cite[Proposition 1.6]{CP93}. 

There are many examples of hyperbolic groups, including virtually free groups, cocompact finitely generated Fuchsian groups, small cancellation $1/6$ groups~\cite[\S 0.2A]{G87}, fundamental groups of compact Riemannian manifolds of negative sectional curvature~\cite[Chapitre 3]{GH90}. By~\cite{BS00}, a hyperbolic group $\Gamma$ with its word metric is roughly similar to a convex subset of some real hyperbolic space $\H^m$, that is, there exist a mapping $\Psi$: $\Gamma \to \H^m$ and $\theta > 0$, $C > 0$ such that $\left| \theta d_{\H}(\Psi(x),\, \Psi(y)) - d(x, \,  y) \right| \leq C$ for all $x$, $y \in \Gamma$, where $d_{\H}$ is the hyperbolic distance on $\H^m$.

Clearly the free groups over $q \ge 2$ symbols are $0$-hyperbolic in both Rips and Gromov senses, and in this case the Gromov product $\gp[z]{x}{y}$ is precisely the distance between $z$ and the geodesic segment $[x, \, y]$.
This remains ``roughly'' true for general $\delta$-hyperbolic group (\cite[Proposition 1.5]{CP93}):
\begin{equation}
  \label{e:roughtree}
  \gp[z]{x}{y} \leq  d(z,\, [x, \, y])  \leq \gp[z]{x}{y} + 4 \delta.
\end{equation}

\bigskip

A hyperbolic group $\Gamma$ is said to be \emph{nonelementary} if it is neither finite nor virtually $\Z$ (i.e., does not contain $\Z$ as a subgroup of finite index). In this case, $\Gamma$ has an exponential volume growth, that is, the \emph{volume entropy} $v$ defined by~\eqref{e:vrate1} is positive.
By~\cite[Theorem 7.2]{C93}, this growth is purely exponential, i.e., there is a constant $C > 1$ such that
\begin{equation}
  \label{e:vgrowth}
  C^{-1} \e^{v n} \leq \left| B(e,\, n) \right| \leq C \e^{v n},\quad \forall\, n\in\N.
\end{equation}
  
Let $\Gamma$ be a nonelementary hyperbolic group. A sequence $(x_n)$ of points in $\Gamma$ is said to \emph{converge at infinity} if $\liminf\limits_{n,\, m \to \infty} \gp[e]{x_n}{x_m} = \infty$, and two such sequences $(x_n)$ and $(y_n)$ are equivalent if $\gp[e]{x_n}{y_n} \to \infty$ as $n \to \infty$. The \emph{hyperbolic boundary} of $\Gamma$, denoted by $\partial \Gamma$, consists of the equivalent classes of sequences converging at infinity. Say a sequence $(x_n)$ converges to $\xi \in \partial \Gamma$ and write $\lim\limits_{n \to \infty} x_n = \xi$ or $x_n \to \xi$ if this sequence belongs to the equivalent class $\xi$.

The Gromov product can be extended to $\partial \Gamma$ by taking limits of Gromov products in $\Gamma$:
\begin{equation}
  \label{e:gpboundary}
  \gp[z]{\xi}{\eta} = \sup \liminf_{m,\, n \to \infty} \gp[z]{x_m}{x_n},
\end{equation}
where the supremum is taken over all sequences $(x_m)$ and $(y_n)$ in $\Gamma$ such that $\xi = \lim\limits_{m \to \infty} x_m$ and $\eta = \lim\limits_{n\to \infty} y_n$. Taking the supremum and $\liminf$ in this definition is, by no means, the only choice. In fact, all four possible choices in the definition of $\gp[z]{\xi}{\eta}$ differ by at most $2 \delta$ (c.f. V\"{a}is\"{a}l\"{a}~\cite[Lemma 5.6]{V05}). Under this extension, we have for every $x$, $y$, $z$, $w \in \Gamma \cup \partial \Gamma$ that
\begin{eqnarray}
\label{e:extend-Gromov-product}
\gp[w]{x}{y} \geq \min \left\{ \gp[w]{x}{z},\, \gp[w]{y}{z} \right\} - 2 \delta.
\end{eqnarray}

As mentioned in the introduction, $\partial \Gamma$ can be metrized following Gromov~\cite[\S 7.2]{G87}. For a real number $a > 1$ sufficiently close to $1$, there is a metric $d_a$ on $\Gamma \cup \partial \Gamma$, called the \emph{visual metric with parameter $a$}, such that the following hold:
\begin{enumerate}[(P1)]
\item The metric $d_a$ induced the canonical boundary topology on $\partial \Gamma$.
\item There is a constant $C \geq 1$ depending only on $\delta$ and $a$ such that
\begin{equation}
  \label{e:visual}
  C^{-1} a^{- \gp[e]{x}{y}}  \leq d_a(x,\, y) \leq C  a^{- \gp[e]{x}{y}},\quad \forall\, x \neq y \in \Gamma \cup \partial \Gamma.
\end{equation}
\end{enumerate}
Furthermore, both $(\Gamma \cup \partial \Gamma,\, d_a)$ and $(\partial \Gamma, \, d_a)$ are compact, and any two visual metrics are H\"older equivalent and therefore they define the same topology.



\begin{thm}[{\cite[Corollary 7.5 and Theorem 7.7]{C93}}]
  \label{T:C93}
  Let $\Gamma$ be a nonelementary hyperbolic group with volume entropy $v$. Then
  $(\partial \Gamma, \, d_a)$ has Hausdorff dimension $v/\log a$.
\end{thm}

The proof of Theorem~\ref{T:C93} replies on the idea of a shadow, which was first introduced by Sullivan~\cite{S79} to study the Patterson--Sullivan measure on hyperbolic spaces.

\begin{defn}
  \label{def:shadow}
  Let $\Gamma$ be a nonelementary hyperbolic group and $x$, $z \in \Gamma$.  The \emph{shadow} $\mho_z(x,\, \kappa)$ cast by $x$ from the light source $z$, with parameter $\kappa > 0$, is defined to be the set of points $\xi \in \partial \Gamma$ such that 
  $\gp[z]{x}{\xi} \geq d(z,\, x) - \kappa$. 
\end{defn}

We write $\mho(x, \, \kappa) = \mho_e(x, \, \kappa)$ for simplicity.
Roughly speaking, a point $\xi \in \mho_z(x,\, \kappa)$ if $x$ comes within distance $\kappa$ of any geodesic from $z$ to $\xi$.
Assume that $\Gamma$ is a nonelementary $\delta$-hyperbolic group and $\kappa>2\delta$. Then for each fixed $n \geq 1$, the shadows $\mho(x,\, \kappa)$ with $|x| = n$ cover $\partial \Gamma$ efficiently. More precisely, there is a constant $N$ such that for any $\xi \in \partial \Gamma$ and any $n \geq 1$ there is at least $1$ and are at most $N$ elements $x$ with the properties that $|x| = n$ and $\xi \in \mho(x,\, \kappa)$; see~\cite[Lemma 2.5.6]{C13}.

\subsection{Random walks on hyperbolic groups}
\label{s:rwhg}

Let $\Gamma$ be a nonelementary hyperbolic group and $\mu$ an admissible symmetric probability measure on $\Gamma$. Here $\mu$ is said to be admissible if its support generates $\Gamma$ as a semigroup. We say that $\mu$ has superexponential tails if for all $K > 1$, $\mu \left( B(e,\, n)^c \right) \leq K^{-n}$ for $n$ large enough. Equivalently, $\mu$ has superexponential tails if, for all $r > 0$, the sum $\sum_{x \in \Gamma} \e^{r |x|} \mu(x)$ is finite. Throughout this paper, we will assume that $\mu$ has superexponential tails.

The random walk $Y=(Y_n)_{n=0}^{\infty}$ starting at $x$ with step distribution $\mu$ is defined as $Y_n = x \xi_1 \xi_2 \cdots \xi_n$, where $\xi_1$, $\xi_2$, $\ldots$ is a sequence of i.i.d. random variables with common distribution $\mu$. It may be viewed as a Markov chain on $\Gamma$ whose transition probability is given by $p(x,\, y) = \mu(x^{-1}y)$. For $x$, $y \in \Gamma$ and $r \geq 0$, define the \emph{(weighted) Green function} by
\[
  G_r(x,\, y) := \sum_{n \geq 0} p_n(x, \, y) r^n,
\]
where $p_n(x, \, y) := \P (Y_n = y)$ is the \emph{heat kernel} of the random walk. Clearly, the Green function and heat kernel are invariant under the left action of $\Gamma$, that is, $G_r(x,\, y) = G_r(e,\, x^{-1}y)$ and $p_n(x, \, y) = p_n(e,\, x^{-1}y)$, where $e$ is the identity element of $\Gamma$.

Let $\rho := \limsup_{n\to\infty} \left\{p_n(x,\, y) \right\}^{1/n}$ be the \emph{spectral radius} of the random walk $Y$, that is, $\rho^{-1}$ is the convergence radius of Green function $G_r(x,\, y)$. Since the random walk is irreducible, $\rho$ is independent of $x$, $y \in \Gamma$. By~\cite[Theorem~7.8]{W00}, $Y$ is $\rho$-transient in the sense that $G_{\rho^{-1}}(x, \, y) < \infty$ for all $x$, $y \in \Gamma$. This and the dominated convergence theorem imply that $G_r(x, \, y)$ is continuous in $r\in [0, \, \rho^{-1}]$.

The so-called Ancona inequalities assert that the Green function $G_r(x, \, y)$ is roughly multiplicative along geodesics. Such inequalities were first developed by Ancona~\cite{A88} for $r < \rho^{-1}$ when identifying the Martin boundary for random walks. Recently, Gou\"ezel~\cite{G14, G15} proved that these inequalities holds uniformly for $r \leq \rho^{-1}$. This uniformity will play an important role in the proof of Theorem~\ref{T:Hr}.

\begin{thm}[Uniform Ancona inequalities, {\cite[Theorem~2.3]{G14}}, {\cite[Theorem~1.3]{G15}}]
  \label{T:ancona}
  Assume that $\mu$ is an admissible symmetric probability measure with superexponential tails on a nonelementary hyperbolic group $\Gamma$. Then there exists a constant $C\geq 1$ such that for any $x$, $z\in \Gamma$ and any $y$
  close to a geodesic segment from $x$ to $z$,
  \begin{equation}
    \label{e:ancona}
    G_r(x, \, z) \leq C G_r(x, \, y) G_r(y, \, z),\ r \in [1,\, \rho^{-1}].
  \end{equation}
The constant $C$ depends only on the distance from $y$ to the geodesic segment connecting $x$ and $z$.
\end{thm}

\begin{thm}[Strong uniform Ancona inequalities, {\cite[Theorem 2.9]{G14}}]
  \label{T:ancona1}
  Under the assumptions of Theorem~\ref{T:ancona}, there are positive constants $C$ and $\eta$ such that for all points $x$, $x'$, $y$, $y'$ whose configurations is approximated by tree as follows:
  
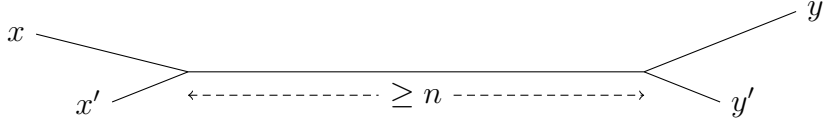
\begin{center}
    \begin{tikzpicture}
    \draw (-5,0.5) -- (-3,0) -- (3,0) -- (5,0.8);
    \draw (-4,-0.4) -- (-3, 0);
    \draw (3,0) -- (4,-0.4);
    \node[left] at (-5,0.5) {$x$};
    \node [left] at (-4,-0.4) {$x'$};
    \node [right] at (5,0.8) {$y$};
    \node [right] at (4,-0.4) {$y'$};
    \node[below] at (0,0) (n) {$\geq n$};
    \path let \p1 = (n.west) in coordinate (a) at (-3,\y1);
    \path let \p1 = (n.east) in coordinate (b) at (3,\y1);
    \draw[<-, densely dashed] (a) -- (n.west);
    \draw[->, densely dashed] (n.east) -- (b);
  \end{tikzpicture}
\end{center}
and for any $r \in [1, \, \rho^{-1}]$,
  \begin{equation}
    \label{e:ancona1}
    \left| \frac{G_r(x, \, y) / G_r(x', \, y)}{G_r(x, \, y') / G_r(x', \, y')} - 1 \right| \leq C \e^{-\eta n}.
  \end{equation}
\end{thm}

We will also use restricted Green functions defined as follows. For a path $\gamma = x x_1 \cdots x_{n-1} y$  of length $n$ from $x$ to $y$, let $w_r(\gamma) = r^n \prod_{i=0}^{n-1} \mu \left( x_i^{-1} x_{i+1} \right)$ with the convention $x_0 = x$ and $x_n = y$. For any subset $A$ of $\Gamma$, define the \emph{restricted Green function} $G_r(x, \, y; \, A) = \sum w_r(\gamma)$ for $0 \leq r \leq \rho^{-1}$, where the sum is over all paths $\gamma = x x_1 \cdots x_{n-1} y$ such that $x_i \in A$ for $1 \leq i \leq n-1$.

The following lemma shows the probability that the random walk does not follow a geodesic is superexponentially small, which was proved in~\cite[Lemma 2.6]{G14} for finitely supported $\mu$ and in~\cite[\S4.4]{G15}
for the general case.

\begin{lemm}[{\cite[Lemma 2.6]{G14}},~\cite{G15}]
  \label{L:gfoutside}
  For every $K > 0$, there is $N > 0$ such that, for all $n \geq N$ and all points $x$, $y$ and $z$ in $\Gamma$ on a geodesic segment (in this order) with $d(x,\, y) \geq n$ and $d(y,\, z) \geq n$,
  \begin{equation}
    \label{e:gfoutside}
    G_{\rho^{-1}} \left( x,\, z; \, B(y,\, n)^c \right) \leq K^{-n}.
  \end{equation}
\end{lemm}

\subsection{Branching random walks on hyperbolic groups}
\label{s:brw}

Let $\mu$ be an admissible symmetric probability on a nonelementary hyperbolic group $\Gamma$ and $\nu$ a probability on $\N$ with mean $\lambda \in [1, \, \infty)$. We recall the definition of the branching random walk ${\rm BRW}(\Gamma, \, \nu, \, \mu)$. At time $n = 0$, one particle is located at $x \in \Gamma$. For any $n \geq 1$, each particle alive at time $n$ dies and gives birth to an independent random number of offspring particles according to the probability $\nu$, each of which independently takes a step on $\Gamma$ according to the random walk $Y$ with step distribution $\mu$.

It is convenient to view the BRW as a tree-indexed random walk (c.f.~\cite{BP94}). Let $\sT$ be a rooted infinite tree. The root is denoted by $\emptyset$ and for a vertex $u$ of $\sT$ let $|u|$ be the graph distance from $u$ to $\emptyset$ and let $u_i$, $0 \leq i \leq |u|$, be the ancestor of $u$ in the $i$-th generation. When $i = |u| - 1$ we denote $u^{-} = u_i$ for simplicity. For $m \geq 0$, let $\sT_m$ be the set of vertices $u$ with $|u| = m$. For $u$, $w \in \sT$, $u \wedge w$ stands for the common ancestor of $u$ and $w$ with the largest generation. Consider a sequence of i.i.d. random variables $\{\eta_u,\, u \in \sT \setminus \{\emptyset\}\}$ with common distribution $\mu$. The random walk on $\Gamma$ indexed by $\sT$ is a collection of $\Gamma$-valued random variables $\{X_u,\, u \in \sT\}$ given by $X_u = x \cdot \prod_{j=1}^{|u|} \eta_{u_j}$. A tree-indexed random walk becomes a
${\rm BRW}(\Gamma, \, \nu, \, \mu)$ if the underlying tree is a Galton--Watson tree induced by the offspring distribution $\nu$.

A BRW is called \emph{recurrent} if almost surely each point in $\Gamma$ is visited infinitely many times and \emph{transient} if almost surely any finite subset is eventually free of particles. As stated in the introduction, we have the following classification in recurrence and transience; see~\cite{BP94} for the sub- and supercritical case and~\cite{GM06} for the critical case. We also refer to~\cite{BZ08} for corresponding results for continuous time BRW and~\cite{GK03} for branching diffusion on Riemannian manifolds. Recall that $\rho$ is the spectral radius of the underlying random walk.
\begin{thm}
  The ${\rm BRW}(\Gamma, \, \nu, \, \mu)$ is transient if and only if $\lambda \leq \rho^{-1}$.
\end{thm}

Let $\sP$ be the set of points in $\Gamma$ that are ever visited by the ${\rm BRW}(\Gamma, \, \nu, \, \mu)$, and $\Lambda$ the \emph{limit set} of the BRW (namely the set of accumulation points of $\sP$).
Clearly for $\lambda > \rho^{-1}$, almost surely $\sP = \Gamma$ and $\Lambda=\partial \Gamma$. In this paper we focus on the transient regime $\lambda \leq \rho^{-1}$ and study the volume growth rate of $\sP$ and the Hausdorff dimension for the limit set $\Lambda$.

\section{Growth for Green functions}
\label{s:rw}

Recall $|x|=d(e,x)$ is the word length of $x \in \Gamma$. For $n \geq 0$ let $\bS_n = \{x\in \Gamma:\ |x| = n\}$ be the sphere with radius $n$ centered at  $e$, the identity element of $\Gamma$. For $r \in [0, \, \rho^{-1}]$,  define
\begin{equation}
  \label{e:Hlambda}
  H_n(r) := \sum_{x \in \bS_n} G_r(e,\, x)\quad \text{and}\ H(r) := \limsup_{n \to \infty} \left\{ H_n(r) \right\}^{1/n}.
\end{equation}
We will see in the next section that $H(r)$ determines the volume growth rate for the BRW on hyperbolic groups.

\begin{thm}
\label{T:Hr}
Assume that $r \in [1, \, \rho^{-1}]$. Then there is a constant $C \geq 1$ such that for every $n \geq 0$ that
\begin{equation}
	\label{e:Hr}
	C^{-1} H(r)^n
\leq
H_n(r)
\leq
C H(r)^n,
\end{equation}
and thus $H(r)=\lim\limits_{n \to \infty} \left\{ H_n(r) \right\}^{1/n}.$ Furthermore, $H(r)$ is continuous and strictly increasing, and satisfies
\begin{equation}
	\label{e:Hrbounds}
        H(r) \leq \e^{v/2}
\end{equation}
with $v$ given by~\eqref{e:vrate1}. 
\end{thm}

By~\eqref{e:Hr}, the growth for $H_n(r)$ is purely exponential. This makes the corresponding dynamical system relatively simple and is crucial when we prove that $H(\lambda)$ has critical exponent $1/2$ at $\rho^{-1}$ in \S\ref{s:exponent}.

We present several lemmas before proving Theorem~\ref{T:Hr}.

\begin{lemm}
  \label{L:subadditivity}
  There exists a constant $C > 0$ such that for every $r \in [1, \, \rho^{-1} ]$, and $m$, $n \geq 0$,
\[
  H_{m + n}(r) \leq C H_m(r) H_n(r).
\]
\end{lemm}

\begin{proof}
  For every $x \in \bS_{m + n}$, choose a geodesics $[e, \, x]$ from $e$ to $x$ and let $y$ be the unique point in $[e, \, x]\cap \bS_m$. Then by the uniform Ancona inequalities (Theorem~\ref{T:ancona}), there exists a constant $C > 0$ depending only on $\Gamma$ such that
\[
  G_{r}(e,\, x)
  \leq
  C G_{r}(e,\, y) G_{r}(y,\, x)
  =
  C G_{r}(e,\, y) G_{r}(e,\, y^{-1} x).
\]
Therefore,
\[
H_{m+n}(r)
=
\sum_{x \in \bS_{m+n}} G_{r}(e,\, x)
\leq
C \sum_{y \in \bS_m} G_{r}(e,\, y) \sum_{z \in \bS_n} G_{r}(e,\, z)
=
C H_m(r) H_n(r).
\]
\end{proof}

\begin{lemm}
  \label{L:supadd}
  There exists a constant $C > 0$ such that for every $r \in [1, \, \rho^{-1} ]$, and $m$, $n \geq 0$,
\[
H_{m+n}(r)
\geq
C H_m(r) H_n(r).
\]
\end{lemm}

\begin{proof}
  By~\cite[Lemma~2.4]{G14}, there exists integer $c_1 > 0$ with the property that for every $x,\ y \in \Gamma$, there is $a \in \Gamma$ of length at most $c_1$ such that $|x a y | \geq |x| + |y|$. Set $\Psi(x,y) = x a y$. By the proof of~\cite[Lemma~2.5]{G14}, there is $c_2 > 0$ such that each point has at most $c_2$ preimages under $\Psi$, that is, $\left| \{ (x, \, y): \ \Psi(x, \, y) = z \} \right| \leq c_2$ for any $z \in \Gamma$.

Now for $x \in \bS_m$ and $y \in \bS_n$, we have that $m + n \leq \Psi(x, \, y) \leq m + n + c_1$ and for some positive constants $c_3$, $c_4$,
  \begin{align*}
    G_r(e, \, x) G_r(e, \, y)
    \leq &
    c_3 G_r(e, \, x) G_r(e, \, a) G_r(e, \, y)
    =
    c_3 G_r(e, \, x) G_r(x, \, x a) G_r(x a, \, x a y) \\
    \leq &
    c_4 G_r(e, \, \Psi(x, \, y)).
  \end{align*}
  As a consequence,
  \begin{align*}
    H_m(r) H_n(r)
    \leq
    c_4 \sum_{x\in \bS_m,\, y \in \bS_n} G_r(e, \, \Psi(x, \, y))
    \leq
    c_4 \sum_{k = m + n}^{m+n + c_1} \sum_{z \in \bS_k} G_r(e, \, z)
    =
    c_4 \sum_{k = 0}^{c_1} H_{m + n + k}(r).
  \end{align*}
  By Lemma~\ref{L:subadditivity}, there are positive constants $c_5$, $c_6$ such that
  \[
    H_m(r) H_n(r)
    \leq
    c_5 \sum_{k=0}^{c_1} H_{m+n}(r) H_k(r)
    \leq
    c_6 H_{m+n}(r).
  \]
\end{proof}

\begin{proof}[Proof of Theorem~\ref{T:Hr}]
By Lemmas~\ref{L:subadditivity} and~\ref{L:supadd}, we have $c_1^{-1} H_m(r) H_n(r) \leq H_{m+n}(r) \leq c_1 H_m(r) H_n(r)$ for some constant $c_1 \geq 1$. Applying Fekete's subadditive lemma,
\[
H(r)
=
\lim_{n \to \infty} \left\{ H_n(r) \right\}^{1/n}
=
\inf_n \left\{ c_1 H_n(r) \right\}^{1/n}
=
\sup_n \left\{ \frac{H_n(r)}{c_1} \right\}^{1/n},
\]
and~\eqref{e:Hr} follows.

For $1 \leq r_1 < r_2 \leq \rho^{-1}$, set $\phi(n) = H_n(r_2) / H_n(r_1)$. Then $\phi(m+n) \leq c_1^2 \phi(m) \phi(n)$. Applying Fekete's subadditive lemma again, we have for any $n_0 > 0$ that
  \[
    \log H(r_2) - \log H(r_1)
    =
    \inf_{n\geq 1} \frac{1}{n} \left[ \log (c_1^2 \phi(n) ) \right]
    \leq
    \frac{\log c_1^2}{n_0} +  \frac{1}{n_{0}} \left[ \log H_{n_0}(r_2) - \log H_{n_0}(r_1) \right].
  \]
Since $H_{n_0}(r)$ is continuous, we prove that $H(r)$ is also continuous in $r \in [1, \, \rho^{-1}]$.

To prove that $H(r)$ is strictly increasing, we first claim that there is $\varepsilon > 0$ such that for all $m \geq 1$, 
\begin{equation}
  \label{e:ldpRW}
  \P \left( \left| Y_m \right| \geq \varepsilon^{-1} m \right)
  \leq
  \rho^m, 
\end{equation}
where $Y_m = \xi_1 \xi_2 \cdots \xi_m$ and $\xi_1$, $\xi_2$, $\ldots$ are i.i.d. random variables with common distribution $\mu$. In fact, if we set $Z_m = \left| \xi_1 \right| + \left| \xi_2 \right| + \cdots + \left| \xi_m \right|$, then $\left| Y_m \right| \leq Z_m$. Since $\mu$ is superexponential,~\eqref{e:ldpRW} follows from the well-known large deviation principle for the one-dimensional random walk $(Z_m)$; see for example~\cite[\S2.2]{DZ10}.

Now for any $1 \leq r_1 < r_2 \leq \rho^{-1}$ and $x \in \bS_n$ we have that 
\begin{align*}
  H_n(r_1)
  =&
     \sum_{x \in \bS_n} \sum_{m=0}^\infty r_1^m p_m (e, \, x) \\
  =&
     \sum_{x \in \bS_n} \sum_{m=0}^{\varepsilon n} r_1^m p_m(e, \, x) + \sum_{x \in \bS_n} \sum_{m = \varepsilon n + 1}^\infty r_1^m p_m(e, \, x) \\
  \leq &
         \sum_{m=0}^{\varepsilon n} r_1^m \P \left( |Y_m| \geq \varepsilon^{-1} m \right)
         + \left( \frac{r_1}{r_2} \right)^{\varepsilon n} \sum_{x \in \bS_n} \sum_{m = \varepsilon n + 1}^\infty r_2^m p_m (e, \, x) \\
  \leq &
         \sum_{m=0}^{\varepsilon n} r_1^m \rho^m + \left( \frac{r_1}{r_2} \right)^{\varepsilon n}  \sum_{x \in \bS_n} G_{r_2}(e, \, x) \\
  \leq & \left( 1 - r_1 \rho \right)^{-1} + \left( \frac{r_1}{r_2} \right)^{\varepsilon n} H_n(r_2). 
\end{align*}
Thus we have that $r_1^{-\varepsilon} H(r_1) \leq r_2^{-\varepsilon} H(r_2)$. In particular, $H(r)$ is strictly increasing in $[1, \, \rho^{-1}]$. 

It remains to prove that $H(r) \leq \e^{v/2}$. By the Cauchy--Schwarz inequality,
\[
\left[ \sum_{x \in \bS_n} G_r(e, \, x) \right]^2
 \leq
\left| \bS_n \right| \sum_{x \in \bS_n} G_r(e, \, x)^2.
\]
By~\cite[Proposition 1.9]{GL13},
\[
c(r) := \sum_{z \in \Gamma} G_r(e, \, z)^2 = \frac{\partial}{\partial r} \left[ r G_r(e, \, e) \right] < \infty
\]
for $r < \rho^{-1}$. It follows that
\[
H(r)
 \leq
\limsup_{n \to \infty} \left[  c(r) | \bS_n| \right]^{\frac{1}{2n}}
  =
 \e^{v/2}, \quad 1 \leq r < \rho^{-1}.
\]
Using the left continuity of $H(r)$ at $r = \rho^{-1}$, we complete the proof.
\end{proof}

\begin{remark}
  \begin{enumerate}[(i)]
  \item In the proof of~\eqref{e:Hrbounds}, the hyperbolicity of $\Gamma$ is only used in the last line. Thus, on any non-amenable group,~\eqref{e:Hrbounds} holds true for every $1 \leq r < \rho^{-1}$.
  \item If the random walk is nearest-neighbor, that is, $\mu$ is supported on the generating set, then we have
    \[
      G_{r_1}(e, \, x) = \sum_{m = n}^\infty r_1^m p_m(e, \, x) \leq \left( \frac{r_1}{r_2} \right)^n G_{r_2}(e, \, x), 
    \]
    and hence $r_1^{-1} H(r_1) \leq r_2^{-1} H(r_2)$. In particular, $H(r) \geq r$ for $r \in [1, \, \rho^{-1}]$. 
  \end{enumerate}
\end{remark}


\section{Volume growth rate for BRW}
\label{s:vgr}

In this section, we will prove that the volume growth rate for ${\rm BRW}(\Gamma, \, \nu, \, \mu)$ on a nonelementary hyperbolic group $\Gamma$ coincides with $H(\lambda)$, the growth rate for Green function of base random walk $Y$ investigated in the previous section.

Assume ${\rm BRW}(\Gamma, \, \nu, \, \mu)$ starts at $e$. For $x \in \Gamma$ denote by $Z_x$ the number of particles in the BRW that ever visit $x$. Let $\sP := \{x \in \Gamma:\ Z_x \geq 1\}$, that is, the set of points in $\Gamma$ that are ever visited by the BRW. Denote by $M_n$ the cardinality of $\sP_n := \sP \cap \bS_n$.

\begin{thm}\label{T:growth}
For $\lambda \in [1, \, \rho^{-1}]$, $H(\lambda) = \limsup_{n \to \infty} M_n^{1/n}$ almost surely.
\end{thm}

The proof of the upper bound for $\limsup M_n^{1/n}$ in Theorem~\ref{T:growth} is quite easy. In fact, by the well-known many-to-one formula, we have that $\E[Z_x] = G_{\lambda}(e,\, x)$. Therefore,
\begin{equation}
  \label{e:upper}
  1 \leq \E [M_n] \leq \sum_{x \in \bS_n} \E[Z_x] = \sum_{x \in \bS_n} G_{\lambda}(e, \, x) = H_n(\lambda).
\end{equation}
Thus for any $\varepsilon > 0$ we have
\[
  \P \left( M_n^{1/n} \geq H(\lambda) + \varepsilon \right) \leq \frac{\E[M_n]}{\left( H(\lambda) + \varepsilon \right)^n}.
\]
As a direct consequence of the Borel--Cantelli Lemma and~\eqref{e:Hr}, we get the following lemma.

\begin{lemm}
  \label{L:Mnupper}
  For $\lambda \in [1,\, \rho^{-1}]$, $\limsup_{n \to \infty} M_n^{1/n} \leq H(\lambda)$ a.s.
\end{lemm}

In the remainder of this section, we prove the lower bound for $\limsup_{n \to \infty} M_n^{1/n}$ in Theorem~\ref{T:growth} by applying the second moment method.

\subsection{The second moment}
Assume that $\nu$ has a finite second moment, i.e. $\sigma^2 := \sum_{k=1}^\infty k^2 \nu(k)<\infty$, in this subsection.
We start with the estimate for the correlation between $Z_x$ and $Z_y$ for $x$, $y \in \Gamma$.

\begin{lemm}
  \label{L:Zxy}
  Assume $\sigma^2 < \infty$. Then for $x$, $y \in \Gamma$,
  \begin{equation}
    \label{e:Ezxy}
    \E \left[ Z_x Z_y \right] \leq \sigma^2 \sum_{z \in \Gamma} G_{\lambda}(e, \, z) G_{\lambda}(z, \, x) G_{\lambda}(z, \, y).
  \end{equation}
\end{lemm}

\begin{proof}
  Note that $Z_x = \sum_{m = 0}^{\infty} \sum_{u \in \sT_m} \id_{\{X_u
    = x \}}$. For $m$, $n \ge k$, conditioned on $|u \wedge w| = k$, the expectation of number of pairs $(u, \, w)$ with $u \in \sT_m$ and $w \in \sT_n$ is at most $\sigma^2 \lambda^{m + n - k}$.  Then
  \begin{align*}
    \E \left[ Z_x Z_y \right]
    =& \E \left[ \sum_{m,\, n=0}^{\infty} \sum_{k=0}^{m \wedge n} \sum_{z \in \Gamma} \sum_{u \in \sT_m} \sum_{w \in \sT_n} \id_{\left\{ X_u = x,\, X_w = y,\, X_{u \wedge w} = z,\, | u \wedge w | = k \right\} } \right] \\
    \leq & \sigma^2 \sum_{k=0}^{\infty} \sum_{m=k}^{\infty} \sum_{n=k}^{\infty} \sum_{z \in \Gamma} \lambda^k p_k(e, \, z) \lambda^{m-k} p_{m-k}(z, \, x) \lambda^{n-k} p_{n-k}(z, \, y) \\
    =& \sigma^2 \sum_{z \in \Gamma} G_{\lambda}(e, \, z) G_{\lambda}(z, \, x) G_{\lambda}(z, \, y).
  \end{align*}
\end{proof}

\begin{cor}
  \label{C:EMnlower}
  Assume $\lambda \in [1,\, \rho^{-1})$ and $\sigma^2 < \infty$. There is a constant $c > 0$ such that $\E[M_n] \geq c \, H(\lambda)^n$ for all $n \geq 0$.
\end{cor}

\begin{proof}
Recall from~\cite[Proposition 6.6]{LP16} that $p_m(x,\, y) \leq \rho^m$. As a consequence,
\[
  G_{\lambda}(z,\, x)^2
\leq
2 \sum_{m=0}^{\infty} \sum_{k=m}^{\infty} \lambda^{m + k} p_m(z,\, x) p_k(z,\, x)
\leq
\frac{2}{1 - \lambda \rho} \sum_{m=0}^{\infty} \left( \lambda \rho \right)^m \lambda^m p_m(z,\, x).
\]
Therefore, we have from Lemma~\ref{L:Zxy} and semigroup property of the heat kernel that
\[
\E[Z_x^2]
\leq
\frac{2 \sigma^2}{1 - \lambda \rho} \sum_{k=0}^{\infty} \sum_{m=0}^{\infty} \left( \lambda \rho \right)^m \lambda^{m+k} \sum_{z \in \Gamma} p_k(e,\, z) p_m(z, \, x)
\leq
\frac{2 \sigma^2}{(1-\lambda \rho)^2} G_{\lambda}(e,\, x).
\]
By the Paley--Zygmund inequality,
\[
  \P \left( Z_x \geq 1 \right)
\geq
\frac{\left( \E[Z_x] \right)^2}{\E[Z_x^2]}
=
\frac{G_{\lambda}(e,\, x)^2}{\E[Z_x^2]}
\geq
c_1 G_{\lambda}(e, \, x)
\]
for some constant $c_1 > 0$. This corollary follows from Theorem~\ref{T:Hr} and the fact that $\E[M_n] =  \sum_{x \in \bS_n} \P \left( Z_x \geq 1 \right)$.
\end{proof}

The following lemma is crucial for our estimates and will also be used in Section~\ref{s:dim}. We present its proof at the end of this section.

\begin{lemm}
  \label{L:dxy2l}
  For every $\lambda \in [1,\, \rho^{-1})$, there exists a constant $c > 0$ such that for $0 \leq k \leq 2 n$,
  \begin{equation}
    \label{e:dxy2l}
    \sum_{\stackrel{x, \, y \in \bS_n}{d(x, \, y) = k}} \sum_{z \in \Gamma} G_{\lambda}(e, \, z) G_{\lambda}(z, \, x) G_{\lambda}(z, \, y)
    \leq
    c \, H(\lambda)^{n + k / 2}.
  \end{equation}
\end{lemm}

Now we are ready to establish the upper bound for $\E[M_n^2]$.

\begin{cor}
  \label{C:varHn}
Assume $\lambda \in [1, \, \rho^{-1})$ and $\sigma^2 < \infty$. There is a constant $c > 0$ such that $\E[M_n^2] \leq c \, \left( \E[M_n] \right)^2$ for every $n \geq 0$.
\end{cor}

\begin{proof}
  By Lemmas~\ref{L:Zxy} and~\ref{L:dxy2l}, we have that $\E[M_n^2] \leq c_1\, H(\lambda)^{2n}$ for some $c_1 > 0$. Applying Corollary~\ref{C:EMnlower} we complete the proof of this result.
\end{proof}

\subsection{Proof of Theorem~\ref{T:growth}}

\begin{proof}[Proof of Theorem~\ref{T:growth}]
  By Lemma~\ref{L:Mnupper}, it suffices to prove that $\limsup_{n \to \infty} M_n^{1/n} \geq H(\lambda)$ almost surely.

  We first assume that $\lambda \in [1,\, \rho^{-1})$ and $\sigma^2 < \infty$. By Corollary~\ref{C:EMnlower}, there is $c_1 > 0$ such that $\E[M_n] \geq c_1 H(\lambda)^n$. Applying the Paley--Zygmund inequality and Corollary~\ref{C:varHn},
  \[
    \P \left( M_n \geq 2^{-1} c_1 H(\lambda)^n \right)
    \geq
    \P \left( M_n \geq 2^{-1} \E [M_n] \right)
    \geq
    \frac{\left( \E[M_n] \right)^2}{4 \E[M_n^2]}
    \geq
    \frac{c_{2}}{4}
  \]
  for some $c_2 > 0$. Therefore, with probability at least $c_2/4 > 0$, the events
  \[
    \left\{ M_n^{1/n} \geq \left( c_1 / 2 \right)^{1/n} H(\lambda) \right\}
  \]
  occur for infinitely many $n$. Applying Lemma~\ref{L:nonrandom} below, we prove the theorem for $\lambda \in [1, \, \rho^{-1})$.

It remains to consider the case $\lambda = \rho^{-1}$ or $\sigma^2 = \infty$. 
For any small enough $\varepsilon > 0$, we may construct another probability measure $\nu'$ on $\N$ with mean $\lambda - \varepsilon$ and finite second moment such that $\nu'$ is stochastically dominated by $\nu$. Let $M_n'$ be the number of vertices $x \in \bS_n$ that are visited by ${\rm BRW}(\Gamma, \, \nu', \, \mu)$ starting at $e$. It is easy to see that $M_n'$ is also stochastically dominated by $M_n$. Therefore,
\[
\P \left( \limsup_{n \to \infty} M_n^{1/n} \geq H(\lambda - \varepsilon) \right)
\geq
\P \left( \limsup_{n \to \infty} (M_n')^{1/n} \geq H(\lambda - \varepsilon) \right)
=
1.
\]
Since $H$ is continuous in $[1, \, \rho^{-1}]$, we obtain that $\limsup_{n \to \infty} M_n^{1/n} \geq H(\lambda)$ a.s., which complete the proof of this theorem.
\end{proof}

\begin{lemm}
\label{L:nonrandom}
  For every $\lambda \in [1, \, \rho^{-1}]$, the limit $\limsup_{n \to \infty} M_n^{1/n}$ is almost surely a constant.
\end{lemm}




\begin{proof}
  The proof is standard and might be known in the literature.  Here we spell out the details for readers' convenience.

  Let $X = (X_u, \, u \in \sT)$ be the BRW we are considering. For $u \in \sT_1$, let $\sT^{(u)}$ be the subtree of $\sT$ consisting of the children of $u$ and $X^{(u)} = \{X_{uv}, \, v \in \sT^{(u)}\}$. Conditional on $\{X_u,\, u \in \sT_1\}$, the processes $X^{(u)}$, $u \in \sT_1$ are independent BRWs starting respectively at $X_u$. Let $\eta$ and $\eta_u$, $u \in \sT_1$ be respectively the volume growth rate for the traces of the BRWs $X$ and $X^{(u)}$. Clearly, $\eta_u$ does not depend on the location of $X_u$ and conditioned on $\sT_1$, $\eta_u$ has the same distribution as that of $\eta$. Let $\varphi(s) = \sum_{k=1}^{\infty} s^k \nu(k)$ be the generating function of $\left| \sT_1 \right|$. Then for every $h \geq 0$, we have from the fact
\[
    \eta = \max_{u \in \sT_1} \eta_u,
  \]
  that
  \[
    \P(\eta \leq h) = \sum_{k=0}^\infty \P (|\sT_1| = k) \left[ \P(\eta \leq h) \right]^k
    = \varphi( \P(\eta \leq h) ).
  \]
  Therefore the probability $\P(\eta \leq h)$ is either $0$ or $1$ and $\eta$ is almost surely a constant.
\end{proof}

\subsection{Proof of Lemma~\ref{L:dxy2l}}

\begin{proof}[Proof of Lemma~\ref{L:dxy2l}]
By the uniform Ancona inequalities, it suffices to prove this lemma for $k = 2l$ with $0 \leq l \leq n$. For $x$, $y \in \bS_n$ with $d(x, \, y) = 2l$, let $w$ be the projection of $e$ to the geodesic $[x, \, y]$. Then from~\eqref{e:roughtree},
\begin{eqnarray}
\label{e:|w|}
n-l
=
\gp[e]{x}{y} \leq d(e, \, w)=\vert w\vert
\leq
\gp[e]{x}{y} + 4\delta
=
n-l+4\delta.
\end{eqnarray}
Using~\eqref{e:roughtree} again,
\[
  d(e,\,w) - 4 \delta \leq \gp[e]{x}{w} \leq d(e,\,w),
  \quad
  d(e,\,w) - 4 \delta \leq \gp[e]{y}{w} \leq d(e,\,w),
\]
which together with~\eqref{e:|w|} imply that
\begin{eqnarray}
\label{e:d(x,w)-d(y,w)}
  l - 8 \delta \leq d(x,\, w) \leq l + 8 \delta,
  \quad
  l - 8 \delta \leq d(y,\, w) \leq l + 8 \delta.
\end{eqnarray}

For $z \in \Gamma$,  choose $u := u(z)\in [e, \, w] \cup [w, \, x] \cup [w, \, y]$ so that
\[
d(z, \, u) = \min \{ d(z, \, [e, \, w]),\ d(z, \, [w, \, x]), \ d(z, \, [w, \, y]) \}.
\]
Let $\Omega_1(x, \, y)$, $\Omega_2(x, \, y)$ and $\Omega_3(x, \, y)$ be respectively the set of the points $z \in \Gamma$ such that $u(z)$ belongs to $[e, \, w]$, $[w, \, x]$ and $[w, \, y]$. See Figure~\ref{f:tree} for illustration of relative locations of $e$, $x$, $y$ and $z$.
\begin{figure}[t]
  \centering
  \begin{tikzpicture}
    \draw (0,0) coordinate (A) -- (0,2) coordinate (B) -- (1.5,3.5) coordinate (D);
    \draw (B) -- (-1.5,3.5) coordinate (C);
    \draw (0,0.7) coordinate (E) -- (-1.25,1.25) coordinate (F);
    \node [below right] at (A) {$e$};
    \node [right] at (B) {$w$};
    \node [right] at (D) {$y$};
    \node [left] at (C) {$x$};
    \node [left] at (F) {$z$};
    \node [right] at (E) {$u$};
    \filldraw (A) circle (1pt);
    \filldraw (B) circle (1pt);
    \filldraw (C) circle (1pt);
    \filldraw (D) circle (1pt);
    \filldraw (E) circle (1pt);
    \filldraw (F) circle (1pt);
  \end{tikzpicture}
  \qquad
    \begin{tikzpicture}
    \draw (0,0) coordinate (A) -- (0,2) coordinate (B) -- (1.5,3.5) coordinate (D);
    \draw (B) -- (-1.5,3.5) coordinate (C);
    \draw (-0.7,2.7) coordinate (E) -- (-1.75,2.5) coordinate (F);
    \node [below right] at (A) {$e$};
    \node [right] at (B) {$w$};
    \node [right] at (D) {$y$};
    \node [left] at (C) {$x$};
    \node [left] at (F) {$z$};
    \node [right] at (E) {$u$};
    \filldraw (A) circle (1pt);
    \filldraw (B) circle (1pt);
    \filldraw (C) circle (1pt);
    \filldraw (D) circle (1pt);
    \filldraw (E) circle (1pt);
    \filldraw (F) circle (1pt);
  \end{tikzpicture}
  \caption{Configurations for relative locations of $e$, $x$, $y$ and $z$ in proof of Lemma~\ref{L:dxy2l}. Left: $z \in \Omega_1(x,\, y)$; Right: $z \in \Omega_2(x,\, y)$.}
  \label{f:tree}
\end{figure}
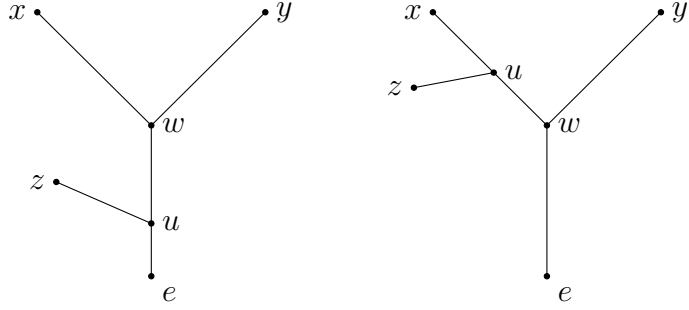

For $z \in \Omega_1(x, \, y)$, we have from~\eqref{e:roughtree} that $\gp[u]{z}{w} = d(z, \, u) - \gp[z]{u}{w} \leq 4 \delta$ and $\gp[w]{u}{x} = d(u, \, w) - \gp[u]{w}{x} \leq 4 \delta$. Combining these two inequalities,
\begin{equation}
  \label{e:uw}
  d(z, \, w) + d(u, \, x) \geq d(z, \, u) + d(w, \, x) + 2 d(u, \, w) - 16 \delta.
\end{equation}

\noindent{\bf Case 1. $d(u, \, w) > 7 \delta$.} By~\eqref{e:uw},
\[
  d(z, \, w) + d(u, \, x) > d(z, \, u) + d(w, \, x) + 2 \delta.
\]
This and~\eqref{e:4points} imply that
\[
d(z, \, w) + d(u, \, x) \leq  d(z, \, x) + d(u, \, w) + 2 \delta.
\]
Using~\eqref{e:uw} again, the above inequality may be rewritten as
 \begin{equation}
   \label{e:tree4}
    d(z, \, x) \geq d(z, \, u) + d(u, \, w) + d(w, \, x) - 18 \delta.
 \end{equation}
\noindent{\bf Case 2. $d(u, \, w) \leq 7 \delta$.} By~\eqref{e:roughtree},
\begin{gather*}
  d(z, \, u) \leq d(z, \, [w, \, x])
  \leq \gp[z]{w}{x} + 4 \delta
  = \frac{1}{2} \left[ d(z, \, x) + d(z, \, w) - d(w, \, x) \right] + 4 \delta, \\
\gp[w]{z}{u} \leq d(w,\, [z,\, u]) \leq d(u,\,w).
\end{gather*}
Hence,
\begin{equation}
  \label{e:tree5}
  \begin{aligned}
    d(z, \, x)
    \geq &
    2 d(z, \, u) + d(w, \, x) - d(z, \, w) - 8 \delta \\
    =&
    d(z,\,u) + d(u,\,w) + d(w, \, x) - 2 \gp[w]{z}{u} - 8 \delta \\
    \geq &
    d(z,\,u) + d(u,\,w) + d(w,\,x) - 2 d(u,\,w) - 8 \delta \\
    \geq &
    d(z,\,u) + d(u,\,w) + d(w,\,x) - 22 \delta.
  \end{aligned}
\end{equation}
From~\eqref{e:tree4} and~\eqref{e:tree5}, in both cases, $u$ and $w$ are close to the geodesic segment connecting $z$ and $x$.

By the uniform Ancona inequalities, there is a constant $c_1>0$ depending only on $\delta$ such that
 \[
 G_\lambda(z, \, x) \leq c_1 \, G_\lambda(z, \, u) G_\lambda(u, \, w) G_\lambda(w, \, x).
 \]
 By the same arguments,
 \[
 G_\lambda(z, \, y) \leq c_1 \, G_\lambda(z, \, u) G_\lambda(u, \, w) G_\lambda(w, \, y).
 \]
 Therefore, combining with~\eqref{e:|w|} and~\eqref{e:d(x,w)-d(y,w)}, we have that there are positive constants $c_2$ and $c_3$ depending only on $\delta$ such that
 \begin{equation}
        \label{e:Omega1}
        \begin{aligned}
                 & \sum_{\stackrel{x, \, y \in \bS_n}{d(x, \, y) = 2l}} \sum_{z\in \Omega_1(x, \, y)} G_\lambda(e, \, z) G_\lambda(z, \, x) G_\lambda(z, \, y) \\
                 \leq &
                 c_2 \sum_{|u| \leq n-l+2\delta} G_\lambda(e, \, u) \sum_{w \in \Gamma} G_\lambda(u, \, w)^2 \sum_{z \in \Gamma} G_\lambda(u, \, z)^2 \sum\limits_{\stackrel{l-2\delta\leq\vert w^{-1}x\vert\leq l+4\delta}{l-2\delta\leq\vert w^{-1}y\vert\leq l+4\delta}} G_\lambda(w, \, x) G_\lambda(w, \, y) \\
                 \leq &
                 c_3 H(\lambda)^{n+l},
        \end{aligned}
 \end{equation}
 where we have used~\eqref{e:Hr} and the fact that $\sum_{w \in \Gamma} G_{\lambda}(e, \, w)^2 < \infty$ for $\lambda < \rho^{-1}$ (c.f.~\cite[Corollary 3.3]{G14}) in the last inequality.

 Now assume $z \in \Omega_2(x, \, y) \cup \Omega_3(x, \, y)$. Without loss of generality, we only need to consider the case $z \in \Omega_2(x, \, y)$.  By the similar arguments, one can prove that $u$ and $w$ are close to the geodesic segments $[e, \, z]$ and $[y, \, z]$.
By~\eqref{e:|w|}--\eqref{e:d(x,w)-d(y,w)}, and the uniform Ancona inequalities, there are positive constants $c_4$ and $c_5$ depending only on $\delta$ such that
 \begin{equation}
   \label{e:Omega2}
   \begin{aligned}
    & \sum_{\stackrel{x, \, y \in \bS_n}{d(x, \, y) = 2l}} \sum_{z\in \Omega_2(x, \, y)} G_\lambda(e, \, z) G_\lambda(z, \, x) G_\lambda(z, \, y) \\
    \leq& 
     c_4 \sum_{n-l\leq \vert w\vert\leq n-l+2\delta} G_{\lambda}(e, \, w) \sum_{u \in \Gamma} G_{\lambda}(w, \, u)^2 \sum_{z \in \Gamma} G_{\lambda}(z, \, u)^2\\
    &\hskip 6cm  \sum_{|u^{-1}x| \leq l+4\delta} G_{\lambda}(u, \, x) \sum_{l-2\delta\leq\vert w^{-1}y\vert\leq l+4\delta} G_{\lambda}(w, \, y) \\
    \leq &
     c_5 H(\lambda)^{n+l}.
   \end{aligned}
 \end{equation}
 Combining~\eqref{e:Omega1} and~\eqref{e:Omega2}, we complete the proof of this lemma.
\end{proof}

\section{Hausdorff dimension of limit set}
\label{s:dim}

In this section, $\mu$ is an admissible, superexponential and symmetric probability on nonelementary hyperbolic group $\Gamma$, and $\nu$ is a probability on $\N$ with mean $\lambda \in [1,\, \rho^{-1}]$. Let $Y$ be the random walk on $\Gamma$ with step distribution $\mu$ and $\rho$ its spectral radius. 
Recall that $(X_u, \, u \in \sT)$ is the branching random walk on $\Gamma$ with offspring distribution $\nu$ and base motion $Y$. Fix a visual metric $d_a$ on $\partial \Gamma$ with parameter $a > 1$ and let $\Lambda$ be the limit set of $(X_u, \, u \in \sT)$ defined as the set of accumulation points on $\partial \Gamma$. 

\begin{thm}
  \label{T:Hausdorff}
  Assume $\lambda \in [1, \, \rho^{-1})$. Then the Hausdorff dimension of $(\Lambda,\, d_a)$ is, with probability $1$, equal to $h(\lambda) := \log_a H(\lambda)$.
\end{thm}

To prove Theorem~\ref{T:Hausdorff}, we first study some properties of the limit set $\Lambda$ in \S\ref{s:limitset}, and then prove the upper bound $\dim_H(\Lambda) \leq h(\lambda)$ in \S\ref{s:dim:up} and the lower bound $\dim_H(\Lambda) \geq h(\lambda)$ in \S\ref{s:dim:low}.

\subsection{Limit set of BRW}
\label{s:limitset}

\begin{lemm}
  \label{L:speedup}
  There is a constant $C > 0$ such that, almost surely, we have $|X_u| \leq C |u|$ for $|u|$ large enough. 
\end{lemm}

\begin{proof}
  For $u \in \sT$ with $|u| = n$, let $\langle {\emptyset = u(0),\, u(1), \ldots,\, u(n) = u} \rangle$ be the geodesic line on $\sT$ connecting $\emptyset$ and $u$. Define $V_u = |X_{u(0)}| + |X_{u(0)}^{-1} X_{u(1)}| + \cdots + |X_{u(n-1)}^{-1} X_{u(n)}|$. Then $(V_u, \, u \in \sT)$ is a branching random walk on the real line. It is well known that the limit $\lim_{n \to \infty} \frac{1}{n} \max_{|u| = n} V_u$ exists and is almost surely a finite constant; see for example Kingman~\cite{K75} and Biggins~\cite{B76}. The lemma follows immediately. 
\end{proof}

Recall that $v$ is the volume entropy of $\Gamma$ defined in~\eqref{e:vrate1}.
\begin{lemm}
  \label{L:speedsub}
  Suppose $\lambda \in [1, \, \rho^{-1})$. Then for any $0 < \ell < -v^{-1} \log (\lambda \rho)$, almost surely there are only finitely many $u \in \sT$ such that $\left| X_u \right| < \ell |u|$. In particular, $\liminf_{n \to \infty} \frac{1}{n} \inf_{|u| = n} \left| X_u \right| \geq - v^{-1} \log (\lambda \rho)$ a.s. 
\end{lemm}

\begin{proof}
  Note that for all $x \in \Gamma$, $p_n(e, \, x) \leq \rho^n$.  We have from~\eqref{e:vgrowth} that
  \begin{align*}
    \E \left[ \sum_{u \in \sT} \id_{\{|X_u| < \ell u\}} \right]
    =
    \sum_{n=0}^\infty \E \left[ \sum_{u \in \sT_n} \id_{\{|X_u| < \ell u\}} \right]
    \leq
    \sum_{n=0}^\infty \lambda^n \sum_{|x| < \ell n} p_n(e, \, x)
    \leq
    \sum_{n=0}^\infty c_1 \lambda^n \rho^n \e^{v \ell n}
    <
    \infty. 
  \end{align*}
  This completes the proof of this lemma. 
\end{proof}

\begin{remark}
  Since for some constant $c>0$, 
  \[
    1 = \sum_{\vert y\vert\leq n} p_n(e, \, y) \leq c \e^{n (v + \log \rho)},
  \]
  we have that
  \[
    \ell < - v^{-1} \log (\lambda \rho) < - v^{-1} \log \rho \le 1
  \]
  for $1<\lambda < \rho^{-1}$.
\end{remark}

\begin{lemm}
  \label{L:bigjump}
  Assume that $\lambda \in [1, \, \rho^{-1}]$. Then for any $\varepsilon > 0$, almost surely there are only finitely many $u \in \sT$ such that $d(X_{u-}, \, X_u) > \varepsilon |X_{u-}|$.  
\end{lemm}

\begin{proof}
  Let $\eta$ be a random variable on $\Gamma$ distributed as $\mu$. Since $\mu$ is superexponential, $c_0 := \E \left[ \e^{s_0 |\eta|} \right] < \infty$. Let $\sG$ be the $\sigma$-algebra generated by $\sT$.  Choose $s_0$ so that $s_0 \varepsilon > \log H(\lambda)$, where $H(\lambda)$ is defined in~\eqref{e:Hlambda}. Then we have for any $u \in \sT$,
  \[
    \P \left( d \left( X_{u-}, \, X_u \right) > \varepsilon \left| X_{u-} \right| \,\big|\, \sG \right)
    =
    \P \left( \left| \eta \right| > \varepsilon \left| X_{u-} \right| \,\big|\, \sG \right)
    \leq
    c_0 \E \left[ \e^{-s_0 \varepsilon \left| X_{u-} \right|} \,\big|\, \sG \right].
  \]
  As a consequence, 
  \begin{align*}
    \E \left[ \sum_{u \in \sT} \id_{\{d \left( X_{u-}, \, X_u \right) > \varepsilon |X_{u-}|\}} \right] 
    \leq 
           c_0 \E \left[ \sum_{u \in \sT} \E \left[ \e^{- s_0 \varepsilon |X_{u-}|} \,\big|\, \sG \right] \right] 
    \leq  c_0 \sum_{k=0}^\infty \lambda \e^{-s_0 \varepsilon k} \E[M_k], 
  \end{align*}
  where $M_k$ is the number of the vertices on the sphere of radius $n$ that are ever visited by the BRW. By Theorem~\ref{T:Hr} and~\eqref{e:upper}, $\E[M_k] \leq c_1 H(\lambda)^k$ for some $c_1 > 0$. Therefore we have that $\E \left[ \sum_{u \in \sT} \id_{\{d \left( X_{u-}, \, X_u \right) > \varepsilon |X_{u-}|\}} \right] < \infty$ and the lemma follows.  
\end{proof}



For $u$ and $w$ in $\sT \cup \partial \sT$, let $d_\sT(u, \, w) = 2^{- \left| u \wedge w \right|}$. Then $d_\sT$ defines a metric on $\sT \cup \partial \sT$. 

\begin{prop}
  \label{P:holder}
Suppose $\lambda \in [1, \, \rho^{-1})$. Fix $0 < \ell_0 < - v^{-1} \log (\lambda \rho)$ and let $\alpha = \ell_0 \log_2 a$. Then, almost surely, $(\sT,\, d_{\sT}) \ni u\mapsto X_u\in (\Gamma, \, d_a)$
defines a H\"older continuous map of index $\alpha$, and this map can be extended to a H\"older continuous map of index $\alpha$ from $(\sT \cup \partial \sT, \, d_{\sT})$ to $(\Gamma \cup \partial \Gamma,\ , d_a)$.
\end{prop}

\begin{proof}
  Let $\ell_0 < \ell < - v^{-1} \log (\lambda \rho)$. By Lemmas~\ref{L:speedsub} and~\ref{L:bigjump}, almost surely there is $n_0 := n_0(\omega)$ such that for all $u \in \sT$ with $|u| \geq n_0$, we have that $|X_u| \geq \ell |u|$ and $d(X_{u-},\, X_u) \leq (\ell - \ell_0) |u|$. Let us fix such an $\omega$ in the rest of the proof. Denote by $\langle {\emptyset = u(0),\, u(1),\, \ldots,\, u(n) = u} \rangle$ the geodesic line on $\sT$ connecting $\emptyset$ and $u$. Then we have for $j \geq n_0$,
  \[
    \gp[e]{X_{u(j)}}{X_{u(j+1)}}
    \geq
    \frac{1}{2} \left[ \ell j + \ell (j+1) - (\ell - \ell_0) j \right]
    \geq
    \ell_0 j.
  \]
  Therefore, for some positive constants $c_i$, $1 \leq i \leq 3$, 
  \[
    d_a\left( X_{u(j)}, \, X_{u(j+1)} \right) \leq c_1 a^{-\gp[e]{X_{u(j)}}{X_{u(j+1)}}} \leq c_2 a^{-\ell_0 j},
  \]
  and consequently, 
  \[
    d_a(X_u, \, X_{u(j)}) \leq \sum_{i=j}^{|u| - 1} d_a(X_{u(i)}, \, X_{u(i+1)}) \leq c_3 a^{-\ell_0 j}.
  \]
  Now for $w \in \sT$ with $|u \wedge w| \geq n_0$, we have that
  \[
    d_a(X_u, \, X_w) \leq d_a(X_u, \, X_{u \wedge w}) + d_a(X_w, \, X_{u \wedge w}) \leq  2 c_3 a^{-\ell_0 |u \wedge w|} = 2 c_3 \left( d_\sT(u, \, w) \right)^\alpha
  \]
  with $\alpha = \ell_0 \log_2 a$. That is, the map $(\sT,\, d_{\sT}) \ni u \mapsto X_u\in (\Gamma, \, d_a)$ is H\"older continuous of index $\alpha$.

  For any $\gamma \in \partial \sT$ and sequence $(u_n)$ in $\sT$ with $u_n \to \gamma$, we have from the H\"older continuity that $\{X_{u_n}\}$ is a Cauchy sequence in $(\Gamma, \, d_a)$ and thus converges to some limit point $\xi \in \partial \Gamma$. Define $X_\gamma = \xi$. If $(w_n)$ is another sequence in $\sT$ such that $w_n \to \gamma' \in \partial \sT$, then
  \[
    d_a \left( X_\gamma, \, X_{\gamma'} \right)
    =
    \lim_{n \to \infty} d_a \left( X_{u_n}, \, X_{w_n} \right)
    \leq
    2 c_3 d_{\sT}(\gamma, \, \gamma').
  \]
  This implies that the map is well-defined and is H\"older continuous of index $\alpha$ on $\sT \cup \partial \sT$. 
\end{proof}


\begin{cor}
  \label{C:Lambda}
  For $1\leq \lambda < \rho^{-1}$,
  \[
    \Lambda = \left\{ X_\gamma:\ \gamma \in \partial \sT \right\} \quad \text{a.s}.
  \]
\end{cor}

\begin{proof}
  By Proposition~\ref{P:holder}, we have for almost all $\omega$ that the map $u \mapsto X_u$ is continuous on $\sT \cup \partial \sT$. Let us fix such an $\omega$. For any $\xi \in \Lambda$, there exists a sequence $(y_n) \subset \sP$ such that $y_n \to \xi$. Choose $u_n \in \sT$ so that $X_{u_n} = y_n$. Then there are $\gamma \in \partial \sT$ and a subsequence $n(k)$ such that $u_{n(k)}$ converges to $\gamma$. The continuity implies that $\xi = X_\gamma$. 
\end{proof}

\begin{remark}
  \label{R:critical-speed}
  Consider the critical case $\lambda = \rho^{-1}$. For $x \in \Gamma$, there is a path from $e$ to $x$ whose probability is bound from below by $c_1^{-|x|}$, and staying close to a geodesic segment from $e$ to $x$. We deduce that $p_n(e, \, x) \leq c_1^{|x|} p_n(e, \, e) \leq c_2 c_1^{|x|} \rho^n n^{-3/2}$ for $n$ large. Using this fact one can slightly modify the arguments of Lemma~\ref{L:speedsub} to prove that, almost surely there is $n_0 := n_0(\omega)$ such that for every $u \in \sT$ with $|u| \geq n_0$, $|X_u| \geq c_3 \log |u|$.

  If the constant $c_3$ can be chosen so that $c_3 \log a > 1$, then one can check, using the same idea as that of Proposition~\ref{P:holder}, that the map $u \mapsto X_u$ is continuously extended to the boundary $\partial \sT$. As a consequence, we still have that $\Lambda = \{X_\gamma:\ \gamma \in \sT\}$ a.s.

  In particular, if $\Gamma$ is virtually free, that is, if $\Gamma$ has a free subgroup of finite index, then the visual parameter $a$ can be chosen arbitrarily large so that $c_3 \log a > 1$. In this case, we have that $\Lambda = \{X_\gamma:\ \gamma \in \sT\}$ a.s. 
\end{remark}

\subsection{Proof for upper bound}
\label{s:dim:up}

Recall that $\sP$ is set of points in $\Gamma$ that are ever visited by particles in ${\rm BRW}(\Gamma,\, \nu, \, \mu)$ and $\sP_n$ the collection of the points $x \in \sP$ with $|x| = n$. To show that the Hausdorff dimension of the limit set $\Lambda$ is at most $h(\lambda)$, it suffices to exhibit, for each $h > h(\lambda)$ and $n \in\N$, a covering $\left\{ J_{nk}, \, k \geq 1 \right\}$ of $\Lambda$, such that $\sum_k  \left( \diam J_{nk} \right)^{h}$ converges to $0$ as $n \to \infty$. Here $\diam A$ is the diameter of a subset $A$ of $(\partial\Gamma, \, d_a)$. In the following lemma we use shadows $\mho(x,\, \kappa)$ cast by $x \in \sP$ with suitable parameters $\kappa$ to construct the coverings of $\Lambda$. 



\begin{lemm}
  \label{L:cover}
 Assume $\lambda\in [1, \, \rho^{-1})$. For every $\varepsilon > 0$ and $n \in \N$, 
  \begin{equation}
    \label{e:coverLambda}
    \Lambda \subseteq \bigcup_{m = n}^\infty \bigcup_{x \in \sP_m} \mho (x, \, \varepsilon m)  \quad \text{a.s.}
  \end{equation}
\end{lemm}

\begin{proof}
Fix a geodesic segment $[e, \, x]$ from $e$ to $x$ for any $x\in\Gamma$. For $x \in \Gamma$ with $|x| \geq n$, let $\pi_n(x)$ be the projection of $x$ on $\bS_n$, that is, $\pi_n(x)$ is the point in $[e, \, x] \cap \bS_n$. For $\varepsilon > 0$ consider the events
\[ A_n
  =
  \left\{\exists \, z \in \sP_n\ \text{such that}\ \sP \cap B(\pi_k(z), \, \varepsilon n) = \emptyset \text{ for some } 1 \le k \le n \right\}.
\]
Fix $K > 0$ with $\varepsilon \log K > v$. By Lemma~\ref{L:gfoutside} and~\eqref{e:vgrowth}, we have that, for sufficiently large $n$,
\begin{equation}
  \label{e:proboutside}
  \begin{aligned}
  \P (A_n)
  \leq &
  \sum_{z \in \bS_n} \sum_{k = 1}^n \P \left( z \in \sP_n\ \text{and}\ \sP \cap B \left( \pi_k(z), \, \varepsilon n \right) = \emptyset \right) \\
  \leq &
  \sum_{z \in \bS_n} \sum_{k = 1}^n G_{\rho^{-1}} \left( e, \, z; \, B(\pi_k(z), \, \varepsilon n)^c \right) \\
  \leq &
  c n \e^{v n} K^{- \varepsilon n},
  \end{aligned}
\end{equation}
where $c$ is a positive constant. By the choice of $K$ we have that $\sum_{n = 1}^{\infty} \P (A_n) < \infty$. Applying the Borel-Cantelli lemma, we have that, for almost all $\omega$, there is $n_0 = n_0(\omega)$ such that for any $n \geq n_0$ and $z \in \sP_n$,
\begin{equation}
  \label{e:projection}
  \sP \cap B \left( \pi_k(z),\, \varepsilon n \right) \neq \emptyset, \quad  \forall 1 \leq k \leq n. 
\end{equation}

There is $\Omega_0 \subset \Omega$ such that $\P \left( \Omega_0 \right) = 1$ and for any $\omega \in \Omega_0$ all the conclusions in Lemma~\ref{L:speedup}--\ref{L:bigjump}, Corollary~\ref{C:Lambda} and~\eqref{e:projection} hold true. Fix $\omega \in \Omega_0$ and let $C$ and $\varepsilon$ be the constants in Lemma~\ref{L:speedup} and~\ref{L:bigjump}. By Corollary~\ref{C:Lambda}, for any $\xi\in\Lambda$, there is a geodesic ray $\emptyset=u_0$, $u_1$, $u_2$, $\ldots$ in $\sT$ so that $\xi = \lim_{n \to \infty} X_{u_n}$. Set $x_n = X_{u_n}$ and fix a constant $0 < \ell_0 < - \frac{\log (\lambda \rho)}{v}$. By Lemma~\ref{L:speedup},~\ref{L:speedsub}, Proposition~\ref{P:holder} and its proof, there is a constant $c_1>0$ such that 
$d_a(x_n, \, \xi) \leq c_1 a^{-\ell_0 n}$ and  $C n \geq |x_n| \geq \ell_0 n$ for $n$ large enough.
This and the fact $d_a(x_n, \, \xi) \geq c_2 a^{-\gp[e]{x_n}{\xi}}$ imply that  
$\gp[e]{x_n}{\xi} \geq \ell_0n - c_3$ for some $c_3$. 
By~\eqref{e:projection}, 
we can choose $y_n \in \sP \cap B \left( \pi_k(x_n), \, \varepsilon |x_n| \right)$, where $k = \lfloor\ell_0 n\rfloor$ and $\lfloor b\rfloor$ is the integer part of $b\in\R$. 
Therefore we have for some positive constants $c_4$ and $c_5$, 
\begin{gather*}
 (\ell_0 - C \varepsilon) n-1 \le |y_n| \le (\ell_0 + C \varepsilon) n,
  \quad
  \gp[e]{x_n}{y_n} \ge (\ell_0 - C \varepsilon) n - c_4,\\
\gp[e]{y_n}{\xi} \geq \min \left\{ \gp[e]{x_n}{y_n}, \, \gp[e]{x_n}{\xi} \right\} - 2 \delta
  \ge (\ell_0 - C \varepsilon) n - c_5
  \ge
  |y_n| - \frac{2 C \varepsilon}{\ell_0 + C \varepsilon} \, |y_n| - c_5,  
\end{gather*}
and thus $\xi \in \mho(y_n, \, 3 C \ell_0^{-1} \varepsilon \, |y_n|).$ This completes the proof of the lemma. 
\end{proof}

\begin{remark}
\label{R:finitely-support}
  If $\mu$ is finitely supported, then by~\cite[Lemma 2.6]{G14}, the right-hand side of~\eqref{e:gfoutside} can be improved to $2^{- \e^{\varepsilon_0 n}}$ for some $\varepsilon_0 > 0$. Using the same idea as the proof of Lemma~\ref{L:cover}, one can obtain that there is a constant $C > 0$ such that
  \[
    \Lambda \subseteq \bigcup_{m = n}^\infty \bigcup_{x \in \sP_m} \mho(x, \, C \log m) \quad \text{a.s.}
  \]
\end{remark}

Recall that $M_n = \left| \sP_n \right|$ and we have proved in Theorem~\ref{T:growth} that $\log H(\lambda) = \limsup\limits_{m \to \infty} \frac{1}{m} \log M_m$ almost surely. 

\begin{proof}[Proof of Theorem~\ref{T:Hausdorff}: upper bound]
For each fixed $h > h(\lambda) = \log H_a(\lambda)$, we can choose $\varepsilon > 0$ so that $(1 - \varepsilon) h \log a > \log H(\lambda)$.  By~\eqref{e:extend-Gromov-product}, we have for any $\xi$ and $\eta$ in $\mho \left( x, \, \varepsilon m \right)$ with $x \in \sP_m$,
\[
  \gp[e]{\xi}{\eta} \ge \min \left\{ \gp[e]{x}{\xi}, \, \gp[e]{x}{\eta} \right\} - 2 \delta
  \ge
  (1 - \varepsilon) m - 2 \delta.
\]
Then we have for some $c > 0$, 
\[
  \diam \mho \left( x, \, \varepsilon m \right)
  \le
  c a^{- (1 - \varepsilon) m},
\]
and therefore
\begin{align*}
  & \sum_{m = n}^\infty \sum_{x \in \sP_m} \left[ \diam \mho \left( x, \, \varepsilon m \right) \right]^h
  \le
  c^h \sum_{m = n}^\infty M_m a^{- (1 - \varepsilon) h m} \\
 & \le
  c^h \sum_{m=n}^\infty \exp \left( - m \left( 1 - \varepsilon \right) h \log a + \log M_m \right),
\end{align*}
which converges to $0$ almost surely. 
Now the desired inequality $\dim_H(\Lambda) \leq h(\lambda)$ follows from Lemma~\ref{L:cover}.
\end{proof}

\begin{remark}
  \label{R:tree-conj}
  If $\Gamma$ is a free group or a free product of finite groups,
  then the inequality $\dim_H (\Lambda) \leq \log_a H(\lambda)$ holds in the critical case $\lambda = \rho^{-1}$. In fact, for any $\varepsilon > 0$ we have from the proof of Lemma~\ref{L:cover} that, for almost all $\omega$, there is $n_0 := n_0(\omega)$ such that~\eqref{e:projection} holds for all $z \in \sP_n$ with $n \geq n_0$. We will prove for such a fixed $\omega$ that
  \begin{equation}
    \label{e:cover4}
    \Lambda \subseteq \bigcup_{m=n}^\infty \bigcup_{z \in \sP_m} \mho \left( z, \, 4 \varepsilon m \right).
  \end{equation}
 For any $\xi \in \Lambda$, we can choose $x \in \sP$ so that $\gp[e]{x}{\xi} > n_0$ and in particular $|x| > n_0$. If $\gp[e]{x}{\xi} \geq (1 -  4 \varepsilon) |x|$, then we have $\xi \in \mho \left( x, \, 4 \varepsilon |x| \right)$. Otherwise, applying~\eqref{e:projection} with $k = \lfloor (1 - 2 \varepsilon) |x| \rfloor$, we can choose $y \in \sP \cap B \left( \pi_k(x), \, \varepsilon |x| \right)$, in particular $\gp[e]{x}{\xi} < (1 - 4 \varepsilon) |x| < |y| < (1 - \varepsilon) |x|$. Since the projection of $y$ on the geodesic segment $[e, \, x]$ has length at least $(1 - 3 \varepsilon) |x| > \gp[e]{x}{\xi} + n_0 \varepsilon$, we have from the tree structure of the Cayley graph of $\Gamma$ that $\gp[e]{y}{\xi} \geq \gp[e]{x}{\xi}$. If $\gp[e]{y}{\xi} \geq (1 - 4 \varepsilon)  |y|$, then we have $\xi \in \mho \left( y, \, 4 \varepsilon |y| \right)$; if not, we can repeat the arguments above. Finally we will get some $z \in \sP$ such that $\xi \in \mho \left( z, \, 4 \varepsilon |z| \right)$ and $|z| > \gp[e]{x}{\xi}$. This implies that~\eqref{e:cover4} holds for every $n \in \N$. By the same arguments as the proof of Theorem~\ref{T:Hausdorff}, we prove that $\dim_H (\Lambda) \leq \log_a H(\rho^{-1})$ in the case $\lambda = \rho^{-1}$.  
\end{remark}

\subsection{Proof for lower bound}
\label{s:dim:low}

Following the same argument as Lemma~\ref{L:nonrandom}, one can prove that $\dim_H(\Lambda)$ is a.s. a constant. By the Frostman's lemma, to prove the lower bound for $\dim_H(\Lambda)$, it suffices to construct a probability measure $\chi$ with support contained in $\Lambda$ such that
\[
  \int_{\Lambda} \int_{\Lambda} d_a(x, \, y)^{-h} \rd \chi(x) \rd \chi(y) < \infty \quad  \text{with positive probability}
\]
for all $h < h(\lambda) = \log_a H(\lambda)$.

Let $\sG$ be the $\sigma$-algebra generated by ${\rm BRW}(\Gamma, \, \nu, \, \mu)$. As before, we let $\sP_n$ be the set of vertices in $\bS_n$ that are ever visited by particles of the BRW and denote $M_n = \left| \sP_n \right|$. Conditioned on $\sG$, we choose for every $n \geq 1$ an element $X_n$ of $\sP_n$ uniformly and let $X_n'$ be an independent copy of $X_n$. Then $\P \left( X_n = X_n' \,\big|\, \sG \right) = M_n^{-1}$. Set
$$
A_n := \left\{ M_n > \frac{1}{2} \E \left[ M_n \right] \right\}
\quad \text{and} \quad
W_n := \E \left[d_a(X_n, \, X_n')^{-h} \id_{\{X_n \neq X_n'\}} \,\big|\, \sG \right].
$$

\begin{lemm}
  \label{L:unifWn}
  For any $0< h < h(\lambda)$, $\E \left[ \id_{A_n} W_n \right]$ is uniformly bounded.
\end{lemm}

\begin{proof}
  Note that $\gp[e]{X_n}{X_n'} = n - d(X_n, \, X_n') / 2$. Therefore we have for $0 < k \leq 2n$,
  \[
    \P \left(\gp[e]{X_n}{X_n'} = \frac{k}{2} \,\big|\,\sG \right)
    =
    \frac{1}{M_n^2} \sum_{d(x, \, y) = 2n - k} \id_{\{x \in \sP_n,\, y \in \sP_n\}}.
  \]
  By Lemmas~\ref{L:Zxy} and~\ref{L:dxy2l},
  \begin{equation}
    \label{e:uppn-t}
        \E \left[ \sum_{d(x, \, y) = 2n - k} \id_{\{x \in \sP_n,\, y \in \sP_n\}} \right]
    \leq
    c_1  H(\lambda)^{2n - k / 2}
  \end{equation}
  for some constant $c_1 > 0$. Applying~\eqref{e:visual} we have that
  \begin{align*}
   \E \left[ \id_{A_n} W_n \right]
    \leq &
    c_2 \E \left[ \id_{A_n} \E \left[\left. a^{h \, \gp[e]{X_n}{X_n'}} \id_{\{X_n \neq X_n'\}} \,\right|\, \sG \right] \right] \\
    = &
    \sum_{k=1}^{2n} c_2 a^{h k / 2} \E \left[ \id_{A_n} \P \left(\left. \gp[e]{X_n}{X_n'} = \frac{k}{2} \, \right|\, \sG \right)  \right].
  \end{align*}
 By~\eqref{e:uppn-t} and Theorem~\ref{T:Hr}, we have for some $c_3$ that
  \[
    \E \left[ \id_{A_n} W_n \right]
    \leq
    c_3 \sum_{k=1}^{\infty} a^{h k / 2} H(\lambda)^{-k/2}
    <
    \infty,
  \]
  which completes the proof of the lemma.
\end{proof}

We continue the proof of Theorem~\ref{T:Hausdorff}.

\begin{proof}[Proof of Theorem~\ref{T:Hausdorff}: lower bound]
As in the proof of Theorem~\ref{T:growth}, it suffices to prove this theorem under the assumption $\sigma^2 = \sum_k k^2 \nu(k) < \infty$. By the Paley--Zygmund inequality and Corollary~\ref{C:varHn},
  \[
    \P(A_n) \geq  \frac{\left( \E \left[ M_n \right] \right)^2}{4 \, \E \left[ M_n^2 \right]} \geq c_1
  \]
  for some positive constant $c_1$. By Lemma~\ref{L:unifWn}, there is a constant $c_2> 0$ such that $\E \left[ \id_{A_n} W_n \right] \leq c_2$ for every $n \geq 1$. Choose $c_3$ so that $c_3^{-1} c_2 < c_1 / 2$. We have that
  \[
    \P \left( A_n,\, W_n > c_3 \right) \leq c_3^{-1} \E \left[ \id_{A_n} W_n \right] \leq c_3^{-1} c_2
  \]
  and hence
  \begin{equation}
    \label{e:posW}
    \P \left( A_n,\, W_n \leq c_3 \right)
    =
    \P \left( A_n \right) - \P \left( A_n,\, W_n > c_3 \right)
    \geq
    c_1/2.
  \end{equation}
Since $\Gamma \cup \partial \Gamma$ is compact, there is a subsequence $\{n_k\}$ such that $X_{n_k}$ (resp. $X_{n_k}'$) converges weakly to some random variable $\eta$ (resp. $\eta'$) on $\partial \Gamma$. By the Fatou's lemma, $\E \left[ d_a (\eta,\, \eta')^{-h} \,\big|\, \sG \right] \leq \liminf_{k \to \infty} W_{n_k}$. By~\eqref{e:posW}, we have that
  \[
    \P \left( \limsup_{k \to \infty} A_{n_k} \cap \left\{ W_{n_k} \leq c_3 \right\} \right) \geq c_1 / 2 > 0.
  \]
Furthermore, for $\omega \in \limsup_{k \to \infty}  A_{n_k} \cap \left\{ W_{n_k} \leq c_3 \right\}$, we have $\liminf_{k \to \infty} W_{n_k}(\omega) \leq c_3$ and hence
  \[
    \E \left[ d_a(\eta,\, \eta')^{-h} \,\big|\, \sG \right](\omega) \leq c_3.
  \]
  Let $\chi$ be the conditional distribution of $\eta$ given $\sG$. Then
  \[
    \int_{\Lambda(\omega)} \int_{\Lambda(\omega)} d_a(x,\, y)^{-h} \rd \chi(x) \rd \chi (y) \leq c_3.
  \]
  By the comments at the beginning of this subsection, we prove the lower bound $\dim(\Lambda) \geq h(\lambda)$ for $\lambda\in [1, \, \rho^{-1})$. 
\end{proof}

\section{Critical exponent}
\label{s:exponent}

To complete proving Theorems~\ref{T:main1} and~\ref{thm1.1}, we show in this section that the function $H(\lambda)$ has critical exponent $1/2$ at the critical point $\rho^{-1}$. In \S\ref{s:sd} we review the thermodynamic formalism associated to an automatic structure of the underlying group (see~\cite{CF10, G14} for more details). Using this machinery, we express the function $\log H(\lambda)$ as the pressure of a transfer operator defined by a certain potential and prove the assertion for critical exponent in \S\ref{s:exponentH}.

\subsection{Symbolic dynamics}
\label{s:sd}

Let $S$ be a finite symmetric generating set of the group $\Gamma$. An \emph{automaton} is a finite direct graph $\sA = (V, \, E, \, s_{*})$ with a distinguished vertex $s_{*}$ as the initial state, and a labeling on edges by generators $\alpha$: $E \to S$. For a directed path $\gamma$ in $\sA$, one can associate a path $\alpha(\gamma)$ in the Cayley graph of $\Gamma$ by multiplying successively the generators read along the edges of $\gamma$. Denote by $\alpha_{*}(\gamma)$ the terminus of $\alpha(\gamma)$.

\begin{defn}
Say that a finitely generated group $\Gamma$ has a \emph{strongly Markov automatic structure} if there is an automaton $\sA = (V, \, E,\, s_{*})$ having the following properties:
  \begin{enumerate}[(i)]
  \item Every $v \in V$ is accessible from the initial state $s_{*}$.
  \item For every directed path $\gamma$ in $\sA$, the path $\alpha(\gamma)$ is a geodesic in $\Gamma$.
  \item The map $\alpha_{*}$ is a bijection of the set of paths starting at $s_{*}$ onto $\Gamma$.
  \end{enumerate}
\end{defn}

By~\cite{C84}, every hyperbolic group admits a strongly Markov automatic structure. In what follows, we fix such an automaton $\sA$ for $\Gamma$. Let $\Sigma^{*}$ (resp. $\Sigma$) be the set of finite (resp. semi-infinite) paths in $\sA$, and $\overline{\Sigma} = \Sigma^{*} \cup \Sigma$. Define the metric $d(\omega, \, \omega') = 2^{-n}$ in $\overline{\Sigma}$, where $n$ is the first time $\omega$ and $\omega'$ diverge. Under this metric, $\Sigma^{*}$ is a dense open subset of the compact space $\overline{\Sigma}$.

By definition, the map $\alpha_{*}$ gives a bijection from the set of paths starting from $s_{*}$ of length $n$ to the sphere $\bS_n$ in $\Gamma$. Moreover, it may be extended naturally to a continuous map from $\overline{\Sigma}$ to $\Gamma \cup \partial \Gamma$.

Let $\sigma$: $\overline{\Sigma} \to \overline{\Sigma}$ be the \emph{(left) shift}, i.e., the map defined by deleting the first edge of a path. For any real-valued H\"older continuous function $\varphi$: $\overline{\Sigma}
\to \R$ (called a potential), we define the transfer operator $\sL_{\varphi}$ acting on the set of continuous functions by
\[
  \sL_{\varphi} f(\omega)
  :=
  \sum_{\sigma(\omega') = \omega}  \e^{\varphi(\omega')} f(\omega'),
\]
where for $\omega = \emptyset$ we only consider the nonempty preimages of the shift $\sigma$. The transfer operator $\sL_{\varphi}$ encodes the Birkhoff sum $S_n \varphi(\omega) := \sum_{j=0}^{n-1} \varphi(\sigma^j \omega)$ in the form
\[
  \sL_{\varphi}^n f(\omega) = \sum_{\sigma^n \omega' = \omega} \e^{S_n \varphi(\omega')} f(\omega').
\]
We are mainly interested in the asymptotics of such sums, which is closely related to the spectrum of $\sL_{\varphi}$ described as follows.

The most fundamental case is where the graph $\sA$ is \emph{topological mixing}, i.e., for two arbitrary vertices $a$, $b$ in $\sA$, for every large enough $n$, there is a path of length $n$ from $a$ to $b$. In this case, the spectrum of $\sL_{\varphi}$ is described by the Ruelle-Perron-Frobenius theorem (Ruelle~\cite{R04}, Bowen ~\cite[Theorem 1.7]{B08}, and Parry and Pollicott~\cite[Theorem 2.2]{PP90}). If the graph $\sA$ is just \emph{recurrent} in the sense that every vertex is accessible from every other vertex, then there is a minimal period $p > 1$ such that the length of any loop is a multiple of $p$. The set of vertices of $\sA$ is the union of $p$ distinct subsets $V_j$, and any edge emanating from a vertex in $V_j$ has the endpoint in $V_{j+1}$ for every $j \in \Z / p \Z$. This decomposition is called a cyclic decomposition of $V$.

When $\sA$ is not even recurrent, one can decompose $\sA$ into components and associate to each component $\sC$ the restriction of $\varphi$ to paths staying in $\sC$. The resulting transfer operator $\sL_{\sC}$ has finitely many eigenvalues of maximal modulus $\e^{\Pr_{\sC}(\varphi)}$ for some real number $\Pr_{\sC}(\varphi)$ (which are called \emph{pressure}), and they are all simple and isolated. Let $\Pr (\varphi) := \max_{\sC} \Pr_{\sC}(\varphi)$ be the maximum of pressure over all components. A component $\sC$ is said to be maximal if $\Pr_{\sC}(\varphi) = \Pr(\varphi)$. The potential $\varphi$ is said to be \emph{semisimple} if there is no directed path between any two distinct maximal components. The following lemma provides a criteria to determine whether a potential is semisimple or not. Let $E_{*}$ be the set of edges starting from the vertex $s_{*}$ and $\id_{[E_{*}]}$ the indicator function equal to $1$ on paths starting with an edge in $E_{*}$ and $0$ elsewhere. By the definition of transfer operator, we have $\sL_{\varphi}^n \id_{[E_{*}]}(\emptyset) = \sum \e^{S_n \varphi(\omega)}$, where the summation is take over all paths $\omega$ starting at $s_{*}$ of length $n$.

\begin{lemm}[{\cite[Lemma 3.7]{G14}}]
  \label{L:maximal}
  Suppose that there is a path from the initial state $s_{*}$ to successively $k > 0$ different maximal components. Then there is a positive constant $C$ such that
  \[
    \sL_{\varphi}^n \id_{[E_{*}]}(\emptyset) \geq C n^{k-1} \e^{n \Pr(\varphi)}.
\]
\end{lemm}

In the case where the potential $\varphi$ is semisimple, the dominating terms of $\sL_{\varphi}^n$ are fairly well decomposed, as in the following theorem. Denote by $\sH^{\beta}$ the space of  $\beta$-H\"older continuous functions on $\overline{\Sigma}$ with the norm $\Vert \cdot \Vert$.

\begin{thm}[{\cite[Theorem 3.8]{G14}}]
  \label{T:semisimple}
  Suppose the potential $\varphi$ is semisimple. Let $\sC_1$, $\ldots$, $\sC_I$ be the maximal components with corresponding periods $p_i\ (1 \leq i \leq I)$, and consider for each $i$ a cyclic decomposition $\sC_i = \cup_{j \in \Z/p_i \Z} \, \sC_{i, \, j}$. Then there are functions $h_{i, \, j}$ and measures $\nu_{i, \, j}$ with $\int h_{i, \, j} \rd \nu_{i,\, j} = 1$, and positive constants $\varepsilon$ and $C$ such that for every H\"older continuous function $f$,
  \begin{equation}
    \label{e:semisimple}
    \left\Vert \sL_{\varphi}^n f - \e^{n \Pr(\varphi)} \sum_{i = 1}^I \sum_{j = 0}^{p_i - 1} \left( \int f \rd \nu_{i, \, ({j-n \mod p_i})} \right) h_{i, \, j} \right\Vert
    \leq
    C \Vert f\Vert \, \e^{-n \varepsilon} \e^{n \Pr(\varphi)}.
  \end{equation}
The probability measures $\rd \mu_i = \frac{1}{p_i} \sum_{j=0}^{p_i - 1}
  h_{i, \, j} \rd \nu_{i,\, j}$ are $\sigma$-invariant and ergodic.

  Denote by $\sC_{\rightarrow, \, i, \, j}$ the set of edges from which $\sC_{i, \, j}$ is accessible with a path of length in $p_i \N$, and $\sC_{i, \, j,\, \rightarrow}$ the set of edges that can be reached from $\sC_{i,\, j}$ by a path of length in $p_i \N$. The function $h_{i,\, j}$ is bounded from below on paths beginning by an edge in $\sC_{i, \, j, \, \rightarrow}$ and the empty path, and vanishes elsewhere. The support of $\nu_{i, \, j}$ is the set of infinite paths beginning in $\sC_{\rightarrow, \, i, \, j}$ and eventually staying in $\sC_i$.
\end{thm}

The following proposition describes the asymptotic behavior of transfer operator under perturbations of the potential.

\begin{prop}[{\cite[Proposition 3.10]{G14}}]
  \label{P:perturbation}
Let $\varphi \in \sH^{\beta}$ be a semisimple potential with maximal components $\sC_1$, $\ldots$, $\sC_I$ and a spectral description as in Theorem~\ref{T:semisimple}. Then there exist $\varepsilon > 0$ and $C > 0$ such that, for all $\psi$ small enough in $\sH^{\beta}$, there exist functions $h_{i, \, j}^{\psi}$ and measures $\nu_{i, \, j}^{\psi}$ with the same support as $h_{i,\, j}$ and $\nu_{i,\, j}$ respectively, and real numbers $\Pr_i(\varphi + \psi)$ satisfying that for every H\"older continuous function $f$,
\[
  \left\Vert \sL_{\varphi + \psi}^n f - \sum_{i=1}^I \e^{n \Pr_i(\varphi + \psi)} \sum_{j = 0}^{p_i - 1} \left( \int f \rd \nu^{\psi}_{i, \, ({j-n \mod p_i})} \right) h^{\psi}_{i,\, j} \right\Vert
  \leq
  C \Vert f\Vert \, \e^{-n \varepsilon} \e^{n \Pr(\varphi)}.
\]
The maps $\psi \mapsto \Pr_i(\varphi + \psi)$, $\psi \mapsto h^{\psi}_{i, \, j}$ and $\psi \mapsto \nu^{\psi}_{i, \, j}$ are real analytic (in the norm sense) from a small ball around $0$ in $\sH^{\beta}$ to $\mathbb{R}$, $\sH^{\beta}$ and the dual of $\sH^{\beta}$ respectively. Finally, with $\rd \mu_i = \frac{1}{p_i} \sum_{j=0}^{p_i-1}h_{i, \, j}\rd\nu_{i, \, j}$,
\[
  \Pr_i(\varphi + \psi)
  =
  \Pr(\varphi) + \int \psi \rd \mu_i + O(\Vert \psi \Vert^2).
\]
 \end{prop}

Note that the pressures $\Pr_{\sC_i}(\varphi + \psi)$ are not necessarily the same. However, for small enough $\psi$, the pressures of $\varphi + \psi$ on components other than the maximal ones are bounded away from $\Pr(\varphi)$. Consequently, $\varphi + \psi$ is also semisimple and its maximal components appear within those of $\varphi$.

\subsection{H\"older continuity of Green functions}

Define for $r \in [1, \, \rho^{-1}]$ the function $\varphi_r$ on $\Sigma^{*}$ by
\begin{equation}
  \label{e:phir}
  \varphi_r(\omega)
  :=
  \log \left( \frac{G_r(e, \, \alpha_{*}(\omega))}{G_r(e, \, \alpha_{*}(\sigma \omega))} \right).
\end{equation}
As a consequence of the uniform Ancona inequalities, the function $\varphi_r$ is $\beta$-H\"older continuous for some $\beta > 0$. Furthermore, we have from~\cite[Lemma 3.11]{G14} that $\left\Vert \varphi_r - \varphi_{\rho^{-1}} \right\Vert \leq c \left(\rho^{-1} - r \right)^{1/3}$ for some $c > 0$. It is claimed in~\cite[Remark 3.12]{G14} that $1/3$ can be replaced by $1/2$ in this inequality. Here we provide a proof for the sake of completeness.

\begin{lemm}
  \label{L:Hphir}
  The function $\varphi_r$ is $\beta$-H\"older continuous for some $\beta > 0$. Furthermore, there is $C > 0$ such that for any $r \in [1,\, \rho^{-1}]$,
  \begin{equation}
    \label{e:Hphir}
    \left\Vert \varphi_r - \varphi_{\rho^{-1}} \right\Vert
    \leq
    C \, \left(\rho^{-1}-r \right)^{1/2}.
  \end{equation}
\end{lemm}

\begin{proof}
  We follow the idea in the proof of~\cite[Lemma 3.4]{G14}. Consider two paths $\omega$, $\omega' \in \Sigma^{*}$ with $d(\omega,\, \omega') = 2^{-n} < 1$. By definition we have that $\omega_k = \omega'_k$ for $0 \leq k \leq n$. Set $x = \alpha_{*}(\omega)$, $x'=\alpha_*(\omega')$, $a = \alpha_{*}(\omega_0)$ and $y = \alpha_{*}(\omega_n)$. Then $\varphi_r(\omega) = \log \left( G_r(e, \, x) / G_r(a, \, x) \right)$ and $\varphi_r(\omega') = \log \left( G_r(e, \, x') / G_r(a, \, x') \right)$. By symmetry, to bound $\varphi_r(\omega) - \varphi_r(\omega')$, it suffices to consider the function $f(r)$ defined by
\[
f(r)
:=
\log \left( \frac{G_r(e, \, x)}{G_r(a, \, x)} \right) - \log \left( \frac{G_r(e, \, y)}{G_r(a, \, y)} \right).
\]
Therefore the lemma follows if we can prove that there are positive constants $C$ and $\eta$ independent of $x$ and $y$ so that
\[
\left| f'(r) \right|
\leq
C \, \e^{-\eta n} \left( \rho^{-1} - r \right)^{-1/2}, \quad r \in [1, \, \rho^{-1}).
\]

Recall from~\cite[Proposition~1.9]{GL13} that the derivative of $G_r(x, \, y)$ with respect to $r$ is given by
\begin{equation}
  \label{e:derivative}
  \frac{\partial}{\partial r} \left[ r G_r(x, \, y) \right] = \sum_{z \in \Gamma} G_r(x, \, z) G_r(z, \, y).
\end{equation}
Thus we have
\begin{align*}
  f'(r)
  =&
     \frac{1}{r} \sum_{z \in \Gamma} \left\{ \frac{G_r(e, \, z) G_r(z, \, x)}{G_r(e, \, x)} - \frac{G_r(a, \, z) G_r(z, \, x)}{G_r(a, \, x)} - \frac{G_r(e, \, z) G_r(z, \, y)}{G_r(e, \, y)} + \frac{G_r(a, \, z) G_r(z, \, y)}{G_r(a, \, y)} \right\} \\
  =&
     \frac{1}{r} \sum_{z \in \Gamma} \frac{G_r(e, \, z) G_r(z, \, y)}{G_r(e, \, y)} \left(A_1 -  A_2 \right),
\end{align*}
where
\[
  A_1 := \frac{G_r(e, \, x) / G_r(a, \, x)}{G_r(e, \, z) / G_r(a, \, z)} \left(\frac{G_r(e, \, y) / G_r(a, \, y)}{G_r(e, \, x) / G_r(a, \, x)} - 1 \right),
\]
and
\[
  A_2 := \left(\frac{G_r(e, \, x) / G_r(a, \, x)}{G_r(e, \, z) / G_r(a, \, z)} - 1 \right) \left(\frac{G_r(e, \, y) / G_r(z, \, y)}{G_r(e, \, x) / G_r(z, \, x)} - 1 \right).
\]
Let $w$ be the projection of $z$ onto the geodesic segment connecting $e$ and $y$. See Figure~\ref{f:holder} for an illustration.
\begin{figure}[h]
  \centering
  \begin{tikzpicture}
    \filldraw (-1.5,0.5) circle (1pt) node[left] {$e$} -- (0,0) circle (1pt) node[above] {$a$} -- (3,0) circle (1pt) node[above right] {$w$} -- (5,0) circle (1pt) node[above] {$y$} -- (7,0.4) circle (1pt) node[right] {$x$};
    \filldraw (3,1.2) circle (1pt) node[above] {$z$} -- (3,0);
    \filldraw (5,0) -- (6.5,-0.4) circle (1pt) node[right] {$x'$};
    \node[below] at (2,0) (n) {$n - 1$};
    \path let \p1 = (n.west) in coordinate (a) at (0,\y1);
    \path let \p1 = (n.west) in coordinate (b) at (5,\y1);
    \draw[<-, densely dashed] (a) -- (n.west);
    \draw[->, densely dashed] (n.east) -- (b);
  \end{tikzpicture}
  \caption{Illustration of the proof for Lemma~\ref{L:Hphir}}
  \label{f:holder}
\end{figure}
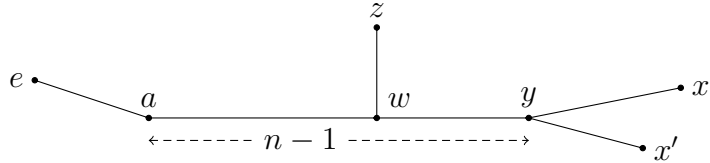
By the strong uniform Ancona inequalities~\cite[Theorem 2.9]{G14} (see also Theorem~\ref{T:ancona1}), there are positive constants $c_1$ and $\eta$ such that
\[
  \left| \frac{G_r(e, \, y) / G_r(a, \, y)}{G_r(e, \, x) / G_r(a, \, x)} - 1 \right|
  \leq
  c_1 \e^{-\eta n}, \quad
  \left| \frac{G_r(e, \, x) / G_r(a, \, x)}{G_r(e, \, z) / G_r(a, \, z)} - 1 \right|
  \leq
  c_1 \e^{-\eta |w|},
\]
and
\[
  \left| \frac{G_r(e, \, y) / G_r(z, \, y)}{G_r(e, \, x) / G_r(z, \, x)} - 1 \right|
  \leq
  c_1 \e^{-\eta (n - |w|)}.
\]
It follows that $|A_1| \leq c_2 \e^{-\eta n}$ and $|A_2| \leq c_2 \e^{-\eta n}$ for some $c_2 > 0$. Applying~\cite[Lemma 3.20 and Theorem 3.1]{G14}, we have that for positive constants $c_3$ and $c_4$,
\[
  \sum_{z \in \Gamma} \frac{G_r(e, \, z) G_r(z, \, y)}{G_r(e, \, y)}
  \leq
  c_3 n \sum_{z \in \Gamma} G_r(e, \, z) G_r(z,\, e)
  \leq
  c_4 n \, (R - r)^{-1/2}.
\]
Therefore we obtain that for positive constants $c_5$ and $c_6$, 
\[
\left| f'(r) \right|
\leq
c_5 \e^{-\eta n} \sum_{z \in \Gamma} \frac{G_r(e,\, z) G_r(z,\, y)}{G_r(e,\, y)}
\leq
c_6 n \e^{-\eta n} \, (R - r)^{-1/2},
\]
which completes the proof of this lemma.
\end{proof}

\subsection{Critical exponent for \texorpdfstring{$H(\lambda)$}{H(λ)}}
\label{s:exponentH}

As mentioned in the last subsection, the function $\varphi_r$ defined in~\eqref{e:phir} is H\"older continuous and can be extended to $\overline{\Sigma}$. Let $\sL_r := \sL_{\varphi_r}$ be the corresponding transfer operator. Then
\begin{align*}
  G_r(e, \, e) \sL^n_r \id_{[E_{*}]} (\emptyset)
  = & G_r(e, \, e) \sum_{\omega = \omega_0 \cdots \omega_{n-1}} \e^{S_n \varphi_r (\omega)} \id_{\{ \omega_0 \in E_{*} \}} \\
  = & \sum_{\omega=\omega_0 \cdots \omega_{n-1}} G_r(e, \, \alpha_{*}(\omega_0 \cdots \omega_{n-1})) \id_{\{\omega_0 \in E_{*}\}},
\end{align*}
where $E_{*}$ is the set of edges starting from $s_{*}$ in $\sA$, and $\id_{[E_{*}]}$ is the function equal to $1$ on the paths starting with an edge in $E_{*}$ and $0$ elsewhere. Since $\alpha_{*}$ is a bijection between the paths of length $n$ starting from $s_{*}$ and $\bS_n$, the function $H_n(r) := \sum_{x \in \bS_n} G_r(e,\, x)$ studied in \S\ref{s:rw} may be expressed as
\begin{equation}
  \label{e:Hnr}
  H_n(r)
  =
  G_r(e, \, e) \sL^n_r \id_{[E_{*}]} (\emptyset).
\end{equation}
Therefore the growth rate $H(r) := \limsup\limits_{n \to \infty} H_n^{1/n}(r)$ is exactly the same as $\e^{\Pr(\varphi_r)}$, the largest eigenvalue of the operator $\sL_r$.

By Theorem~\ref{T:Hr} and~\cite[Lemma 3.7]{G14} (see also Lemma~\ref{L:maximal}), we have that $\sL_r$ is semisimple for every $r \in [1, \, \rho^{-1}]$. Let $\sC_1$, $\ldots$, $\sC_I$ be the maximal components for $\sL_{\rho^{-1}}$ with corresponding periods $p_i$, and take for each $i$ a cyclic decomposition $\sC_i = \bigcup_{0 \leq j \leq p_i - 1} \sC_{i,\, j}$. By~\cite[Proposition 3.10]{G14} (see also Proposition~\ref{P:perturbation}), there is $1 < r_0 < \rho^{-1}$ such that for every $r \in [r_0, \, \rho^{-1}]$, there are functions $h_{i,\, j}^{(r)}$ and measures $\nu_{i,\, j}^{(r)}$ (with the same support as respectively $h_{i,\, j}$ and $\nu_{i,\, j}$ in Theorem~\ref{T:semisimple}) and numbers $\Pr_i(\varphi_r)$ such that
\begin{equation}
  \label{e:Pri}
  \Pr_i(\varphi_r)
  =
  \Pr(\varphi_{\rho^{-1}}) + \int \left( \varphi_r - \varphi_{\rho^{-1}} \right) \rd \mu_i + O(\Vert \varphi_r - \varphi_{\rho^{-1}} \Vert^2)
\end{equation}
and
\begin{equation}
  \label{e:perturbation}
  \left\Vert \sL_r^n f - \sum_{i = 1}^I \e^{n \Pr_i(\varphi_r)} \sum_{j = 0}^{p_i - 1} \left( \int f \rd \nu_{i,\, ({j - n \mod p_i})}^{(r)} \right) h_{i,\, j}^{(r)} \right\Vert
\leq
C \Vert f \Vert \, \e^{-n \varepsilon} \, \e^{n \Pr(\varphi_{\rho^{-1}})},
\end{equation}
for some $\varepsilon > 0$ and $C > 0$, where $\rd \mu_i = \frac{1}{p_i} \sum_{j = 0}^{p_i - 1} h_{i,\, j}^{(\rho^{-1})} \rd \nu_{i,\, j}^{(\rho^{-1})}$. Furthermore, $\Pr_i(\varphi_r)$, $h_{i,\, j}^{(r)}$ and $\nu_{i,\, j}^{(r)}$ are continuous in $r \in [r_0, \, \rho^{-1}]$. Denote by $I(r) = \left\{ 1 \leq i \leq I:\ \Pr_i(\varphi_r) = \max\limits_{1 \leq j \leq I} \Pr_j(\varphi_r) \right\}$.

Define
\[ V_n(r) := \frac{1}{n} \log \left( r H_n(r) \right). \]
We have that
\begin{equation}
  \label{e:Vn'}
  V_n'(r) 
  =
  \frac{1}{r n H_n(r)} \sum_{x \in \bS_n} G_r(e, \, x) \Phi_r(x),
\end{equation}
where
\[ \Phi_r(x) := \sum_{z\in \Gamma} \frac{G_r(e, \, z) G_r(z, \, x)}{G_r(e, \, x)}. \]
Set
\[ \eta(r) := \sum_{y \in \Gamma} G_r(e, \, y) G_r(y, \, e). \]
By~\cite[Lemma 3.20]{G14}, there is a positive constant $C_1$ such that $\Phi_r(x)\leq C_1 \, (1 + |x|) \eta(r)$ for every $r \in [1, \, \rho^{-1}]$ and $x \in \Gamma$. Consequently, we have that $V_n'(r) / \eta(r) \leq 2 C_1$.

The following lemma provides more accurate estimate for the function $\Phi_r(x)$.

\begin{lemm}[{\cite[Lemma 3.23]{G14}}]
  \label{L:Phi}
  There is a family of H\"older continuous functions $\psi_r$ on $\overline{\Sigma}$ for $r \in
  [1,\, \rho^{-1})$ with the following properties:
  \begin{enumerate}[(i)]
  \item As $r \to \rho^{-1}$, $\psi_r$ converges to a function $\psi$ in the H\"older topology.
  \item For any $\omega \in \Sigma^{*}$ of length $n$,
    \begin{equation}
      \label{e:Phi}
      \Phi_r(\alpha_{*}(\omega)) = \eta(r) S_n \psi_r(\omega) + O(\eta(r)),
    \end{equation}
    where $S_n \psi_r$ is the Birkhoff sum $\sum_{k=0}^{n-1} \psi_r \circ \sigma^k$.
  \end{enumerate}
\end{lemm}

In view of~\eqref{e:Vn'} and~\eqref{e:Phi}, we have that
\begin{equation}
  \label{e:Vn'eta}
  \begin{aligned}
      \frac{V_n'(r)}{\eta(r)}
=&
\frac{G_r(e, \, e)}{r H_n(r)} \, \frac{1}{n} \sum_{k = 0}^{n-1} \sL_r^n \left( \id_{[E_{*}]} \cdot \psi_r \circ \sigma^k \right)(\emptyset) + O(n^{-1}) \\
=&
\frac{G_r(e, \, e)}{r H_n(r)} \, \frac{1}{n} \sum_{k = 0}^{n-1} \sL_r^{n-k} \left( \psi_r \sL_r^k \id_{[E_{*}]} \right)(\emptyset) + O(n^{-1}),
  \end{aligned}
\end{equation}
where we have used  $\sL_r(f \cdot g \circ \sigma) = g \sL_r f$ in the second equality. Let $p$ be the least common multiple of $p_i$, $1 \leq i \leq I$. Note that
\[
  \int \psi_r h_{i, \, j}^{(r)} \rd \nu_{i', \, ({ j' + k \mod p_{i'}})}^{(r)}
\]
vanishes except for $i = i'$ and $j \equiv j' + k \mod p_i$. We have from~\eqref{e:perturbation} that
\begin{align*}
  & \sL_r^{n p-k} \left( \psi_r \sL_r^k \id_{[E_{*}]} \right)(\emptyset) \\
  \sim &
     \sum_{i, \, i' \in I(r)} \sum_{j = 0}^{p_i - 1} \sum_{j' = 0}^{p_{i'} - 1} \e^{np \Pr(\varphi_r)} h_{i', \, j'}^{(r)}(\emptyset) \left( \int \id_{[E_{*}]} \rd \nu_{i, \, (j-k\mod p_i)}^{(r)} \right) \left( \int \psi_r h_{i, \, j}^{(r)} \rd \nu_{i', \, (j' + k\mod p_{i'})}^{(r)} \right) \\
  =&
    \e^{np \Pr(\varphi_r)} \sum_{i \in I(r)} \sum_{j, \, j' = 0}^{p_i - 1} \id_{\{j \equiv j' + k \mod p_i\}} \nu_{i, \, j'}^{(r)}([E_{*}]) h_{i, \, j'}^{(r)}(\emptyset) \left( \int \psi_r h_{i, \, j}^{(r)} \rd \nu_{i, \, j}^{(r)} \right)
\end{align*}
as $n \to \infty$. Consequently,
\begin{equation}
  \label{e:Lrnp}
  \begin{aligned}
    & \frac{1}{np} \sum_{k = 0}^{np} \sL_r^{n p-k} \left( \psi_r \sL_r^k \id_{[E_{*}]} \right)(\emptyset) \\
  =&
  \e^{np \Pr(\varphi_r)} \sum_{i \in I(r)} \left( \sum_{j = 0}^{p_i - 1} \nu_{i, \, j}^{(r)}([E_{*}]) h_{i, \, j}^{(r)}(\emptyset) \right) \mu_i^{(r)}(\psi_r) + o \left( \e^{np \Pr(\varphi_r)} \right),
  \end{aligned}
\end{equation}
where $\rd \mu_i^{(r)} = \frac{1}{p_i} \sum_{j = 0}^{p_i - 1} h_{i, \, j}^{(r)} \rd \nu_{i, \, j}^{(r)}$. Recall that $H_n(r) = G_r(e, \, e) \sL_r^n \id_{[E_{*}]}(\emptyset)$. Combining~\eqref{e:perturbation} and~\eqref{e:Lrnp} we prove that for $r_0 \leq r < \rho^{-1}$,
\begin{equation}
  \label{e:limVn'}
  \lim_{n \to \infty} \frac{r V_{n p}'(r)}{\eta(r)}
  =
  \frac{\sum_{i \in I(r)} \left( \sum_{j = 0}^{p_i - 1} \nu_{i, \, j}^{(r)}([E_{*}]) h_{i, \, j}^{(r)}(\emptyset) \right) \mu_i^{(r)}(\psi_r)}{\sum_{i \in I(r)} \sum_{j = 0}^{p_i - 1} \nu_{i, \, j}^{(r)}([E_{*}]) h_{i, \, j}^{(r)}(\emptyset)}
  =: U(r).
\end{equation}
Since
\[
  V_{n p}(\rho^{-1}) - V_{np}(r) = \int_r^{\rho^{-1}} \eta(s) \frac{V_{np}'(s)}{\eta(s)} \rd s,
\]
the dominated convergence theorem and~\eqref{e:limVn'} imply that
\begin{equation}
  \label{e:HRHr}
  \log H(\rho^{-1}) - \log H(r) = \int_r^{\rho^{-1}} \eta(s) U(s) s^{-1} \rd s.
\end{equation}

By~\cite[Theorem 3.1]{G14}, we have that $\eta(r) \sim C_2 \, \left(\rho^{-1} - r \right)^{-1/2}$ as $r \uparrow \rho^{-1}$ for some positive constant $C_2$. Therefore, there is a constant $C_3 > 1$ such that
  \begin{equation}
    \label{e:boundexp}
    C_3^{-1} \, \left(\rho^{-1} - r \right)^{1/2} \leq \log H(\rho^{-1}) - \log H(r) \leq C_3 \, \left(\rho^{-1} - r \right)^{1/2}.
  \end{equation}

  \begin{thm}
    \label{T:exponent}
    There exists a constant $C > 0$ such that
    \begin{equation}
      \label{e:exponent}
      \log H(\rho^{-1}) - \log H(r) \sim C \, \left( \rho^{-1} - r \right)^{1/2} \quad \text{as } r \uparrow \rho^{-1}.
    \end{equation}
  \end{thm}

  \begin{proof}
    Recall that $\log H(r) = \Pr(\varphi_r) = \max_{1 \leq i \leq I} \Pr_i(\varphi_r)$. It suffices to prove that there are constants $C_i$ such that for every $1 \leq i \leq I$,
    \begin{equation}
      \label{e:expi}
      \Pr (\varphi_{\rho^{-1}}) - \Pr_i (\varphi_r) \sim C_i \, \left( \rho^{-1} - r \right)^{1/2} \quad \text{as } r \uparrow \rho^{-1}.
    \end{equation}
    In fact,~\eqref{e:exponent} holds with $C = \min_{1 \leq i \leq I} C_i$, which is positive by~\eqref{e:boundexp}.

    Set $f_r := (\varphi_{\rho^{-1}} - \varphi_r)/\left( \rho^{-1} - r \right)^{1/2}$. By~\cite[Theorem 3.1]{G14}, for any $x$, $y\in \Gamma$, there is $C(x, \, y)>0$ such that
    \[
      \log G_{\rho^{-1}}(x, \, y) - \log G_r(x, \, y) \sim C(x, \, y) \left( \rho^{-1} - r \right)^{1/2} \quad \text{as } r \uparrow \rho^{-1}.
\]
It follows that for any $\omega \in \Sigma^{*}$, $f_r(\omega)$ converges to $f(\omega):= C(e, \, \alpha_{*}(\omega)) - C(e, \, \alpha_{*}(\sigma \omega))$ as $r \uparrow \rho^{-1}$. We will show that $f$ is H\"older continuous and the convergence also holds for $\omega \in \Sigma$. Indeed, we have from Lemma~\ref{L:Hphir} that the family of functions $\left\{ f_r,\, 1 \leq r < R \right\}$ is uniformly bounded in $\sH^{\beta}$ and hence relatively compact in $\sH^{\beta'}$ for $\beta' < \beta$. Let $g$ be any limit point of $f_r$ as $r \uparrow \rho^{-1}$. Then $g \in \sH^{\beta'}$ and $g(\omega) = f(\omega)$ for $\omega \in \Sigma^{*}$. Since $\Sigma^{*}$ is a dense open subset of the compact space $\overline{\Sigma}$, $g$ is uniquely determined. Therefore we have proved that $f_r$ converges to $g$ in $\sH^{\beta'}$ as $r \uparrow \rho^{-1}$. Set $C_i = \int g \rd \mu_i$ for $1 \leq i \leq I$. By~\eqref{e:Pri} and the dominated convergence theorem, we obtain~\eqref{e:expi} and complete the proof of the theorem.
  \end{proof}



\bigskip
\noindent\textbf{Acknowledgements.} We thank Zhan Shi, Xinxin Chen and Shen Lin for valuable discussions. Part of the work has been done while LW and KX were visiting the NYU Shanghai -- ECNU Mathematical Institute. We are grateful to the Institute for hospitality and financial support.

\bibliography{\jobname}{}
\bibliographystyle{ieeetr}

\bigskip


\flushleft{Vladas Sidoravicius\\*
NYU-ECNU Institute of Mathematical Sciences at NYU Shanghai\\*
\& Courant Institute of Mathematical Sciences\\*
New York, NY 10012, USA\\*

\flushleft{Longmin Wang\\*
School of Mathematical Sciences, Nankai University\\*
Tianjin 300071, P. R. China\\*
E-mail address: \texttt{wanglm@nankai.edu.cn}\\*

\flushleft{Kainan Xiang\\*
School of Mathematics and Computational Science, Xiangtan University\\*
Xiangtan City 210000, Hunan Province, P. R. China\\*
E-mail addresses: \texttt{kainan.xiang@xtu.edu.cn}\\*
\hspace{\widthof{E-mail addresses: }}\texttt{kainanxiang@gmail.com}

\end{document}